\documentclass[12pt]{article}
\usepackage{amsfonts}
\usepackage{amssymb}
\usepackage{amsmath}
\usepackage{theorem}

\usepackage{hyperref}

\usepackage[normalem]{ulem}
\usepackage{tikz}
\usetikzlibrary{matrix}
\usetikzlibrary{arrows}

\numberwithin{equation}{section}

\input amssym.def

\makeatletter
\newcommand\qedsymbol{\hbox{$\Box$}}
\newcommand\qed{\relax\ifmmode\Box\else
  {\unskip\nobreak\hfil\penalty50\hskip1em\null\nobreak\hfil\qedsymbol
  \parfillskip=\z@\finalhyphendemerits=0\endgraf}\fi}

\newcommand\subqedsymbol{\hbox{$\triangledown$}}
\newcommand\subqed{\relax\ifmmode\triangledown\else
  {\unskip\nobreak\hfil\penalty50\hskip1em\null\nobreak\hfil\subqedsymbol
  \parfillskip=\z@\finalhyphendemerits=0\endgraf}\fi}
\makeatother

\newenvironment{proof}[1][{}]{\par\noindent Proof{#1}. }{\qed}
\newenvironment{subproof}[1][{}]{\par\noindent Proof{#1}. }{\subqed}

\newcommand{\hotimes}{{\,\hat{\otimes}\,}}

\newcommand{\bfzero}{{\mathbf 0}}

\newcommand{\Sh}{{\mathrm{Sh}}}

\newcommand{\Stub}{{\mathrm{Stub}}}
\newcommand{\Shift}{{\mathrm{Shift}}}

\newcommand{\bul}{{\bullet}}
\newcommand{\al}{{\alpha}}

\newcommand{\mb}{{\mathfrak{b}}}

\newcommand{\md}{{\mathfrak{d}}}
\newcommand{\ms}{{\mathfrak{s}}}

\newcommand{\bs}{{\mathbf{s}}}

\newcommand{\bsi}{{\mathbf{s}^{-1}\,}}

\newcommand{\Om}{{\Omega}}

\newcommand{\si}{{\sigma}}
\newcommand{\ga}{{\gamma}}
\newcommand{\vf}{{\varphi}}
\newcommand{\ve}{{\varepsilon}}

\newcommand{\pa}{{\partial}}

\newcommand{\cF}{{\cal F}}

\newcommand{\cZ}{{\mathcal{Z}}}

\newcommand{\cL}{{\cal L}}

\newcommand{\cM}{{\mathcal{M}}}
\newcommand{\cQ}{{\mathcal{Q}}}

\newcommand{\bbk}{{\Bbbk}}

\newcommand{\D}{{\Delta}}

\newcommand{\ti}[1]{{\tilde{#1}}}

\newcommand{\mMC}{\mathfrak{MC}}

\newcommand{\curv}{{\mathsf{curv}}}

\newcommand{\maps}{\,\colon\,}

\newcommand{\und}[1]{{\underline{#1}}}

\newcommand{\sLie}{{\mathfrak{S}}\mathsf{Lie}_{\infty}}

\newcommand{\tensor}{\otimes}
\newcommand{\xto}[1]{\xrightarrow{#1}}

\newcommand{\MC}{\mathrm{MC}}

\DeclareMathOperator{\id}{\mathrm{id}}

\newtheorem{thm}{Theorem}[section]

\newtheorem{lem}[thm]{Lemma}

\newtheorem{prop}[thm]{Proposition}
\newtheorem{claim}[thm]{Claim}

\theorembodyfont{\rm}

\newtheorem{pty}[thm]{Property}
\newtheorem{remark}[thm]{Remark}

\title{A Version of the Goldman-Millson Theorem for Filtered $L_{\infty}$-Algebras}

\author{Vasily A. Dolgushev and Christopher L. Rogers}

\date{}

\begin{document}

\maketitle

\begin{abstract} 
In this paper we consider $L_{\infty}$-algebras equipped with complete 
descending filtrations.  We prove that, under some mild conditions, 
an $L_{\infty}$ quasi-isomorphism $U: L \to \ti{L}$ induces 
a weak equivalence between the Deligne-Getzler-Hinich (DGH)
$\infty$-groupoids corresponding to $L$ and $\ti{L}$, 
respectively.  This paper may be considered as a modest addition to 
foundational paper \cite{Ezra-infty} by Ezra Getzler. 
\end{abstract}

\section{Introduction}
The Getzler-Hinich construction \cite{Ezra-infty},  \cite{Hinich:1997} 
assigns to every nilpotent $L_{\infty}$-algebra $L$ a Kan complex $\mMC_{\bul}(L)$. 
This Kan complex is a natural generalization of the Deligne groupoid 
\cite{Ezra}, \cite{GM} from dg Lie algebras to nilpotent $L_{\infty}$-algebras. 
Various versions of the Deligne-Getzler-Hinich (DGH) groupoid are used 
in deformation theory \cite{BGNT}, \cite{GM}, \cite{K}, \cite{Lurie-DAG-X}, \cite{Ye}, rational homotopy theory 
\cite{Berglund}, \cite{Lazarev}, \cite{LM}, and derived algebraic geometry 
\cite{Lurie-DAG-X}. In this paper we consider a version of the Kan complex  $\mMC_{\bul}(L)$ 
for $L_{\infty}$-algebras equipped with a complete descending filtration. 
We prove that, under some mild conditions, an $L_{\infty}$ quasi-isomorphism 
$U: L \to \ti{L}$ of such filtered $L_{\infty}$-algebras induces a weak equivalence 
of simplicial sets 
$$
\mMC_{\bul}(U) :  \mMC_{\bul}(L) \to \mMC_{\bul}(\ti{L}).
$$

Let $L$ be a cochain complex of $\bbk$-vector\footnote{In this paper, 
we assume that $\textrm{char}(\bbk) = 0$.} spaces. We recall that an $L_{\infty}$-structure 
$L$ is a degree $1$ coderivation $\cQ$ of the cocommutative coalgebra 
$$
\und{S}(\bsi L) : = \bsi L \oplus S^2(\bsi L) \oplus S^3(\bsi L) \oplus \dots
$$
satisfying the equation 
$$
\cQ^2 = 0\,,
$$
and the condition $\cQ(\bsi \, v)  = - \bsi\, \pa v $,
where $\pa$ is the differential on $L$  and $(\bsi L )^{\bul}= L^{\bul+1}$\,.

An $\infty$-morphism $U$ from an $L_{\infty}$-algebra $(L, \cQ)$ to an $L_{\infty}$-algebra $(\ti{L}, \ti{\cQ})$
is a homomorphism of the cocommutative coalgebras 
\begin{equation}
\label{int-U}
U : \und{S}(\bsi L) \to \und{S}(\bsi \ti{L})
\end{equation}
which intertwines the coderivations $\cQ$ and $\ti{\cQ}$:
$$
U\circ \cQ = \ti{\cQ} \circ U\,. 
$$
  
Let us also recall that   
every coderivation $\cQ$ of  $\und{S}(\bsi L)$ is uniquely determined 
by its composition
$$
\cQ' : = p_{\bsi L} \circ \cQ : \und{S}(\bsi L) \to \bsi L  
$$ 
with the canonical projection $p_{\bsi L} :  \und{S}(\bsi L) \to  \bsi L $\,. 
Similarly, every homomorphism \eqref{int-U} is uniquely determined by 
its composition 
$$
U' : = p_{\bsi \ti{L}} \circ U : \und{S}(\bsi L) \to \bsi \ti{L}\,.  
$$

The restriction of $U'$ to $\bsi L$ gives us a chain map 
\begin{equation}
\label{int-lin-term}
U' \big|_{\bsi L} : \bsi L \to \bsi \ti{L}
\end{equation}
which we call the {\it linear term} of the $\infty$-morphism $U$.  
We recall that an $\infty$-morphism $U$ is an $\infty$ {\it quasi-isomorphism}
if its linear term is a quasi-isomorphism of cochain complexes $\bsi L \to \bsi \ti{L}$\,.

In this paper, we deal with {\it filtered} $L_{\infty}$-algebras $L$, i.e. 
$L_{\infty}$-algebras $L$ equipped with a 
complete descending filtration 
\begin{equation}
\label{int-filtr-L}
L = \cF_{1}L \supset \cF_{2}L \supset  \cF_{3}L  \supset \cdots\,,
\end{equation}
\begin{equation}
\label{int-L-complete}
L =\varprojlim_{k} L/\cF_{k}L
\end{equation}
such that
\begin{equation}
\label{int-Q-filtr-OK}
\cQ' (\cF_{i_{1}} \bsi L \otimes \cF_{i_{2}} \bsi L \otimes \dots \otimes \cF_{i_{m}} \bsi L )
\subset \cF_{i_{1} + i_{2} + \cdots + i_{m}} \bsi  L \quad \forall ~~ m \ge 1.
\end{equation}
Furthermore, we assume that all $\infty$-morphisms $U$ in question are also 
compatible with the filtrations in the sense that 
\begin{equation}
\label{U-filtration-OK}
U' (\cF_{i_{1}} \bsi L \otimes \cF_{i_{2}} \bsi L \otimes \dots \otimes \cF_{i_{m}} \bsi L )
\subset \cF_{i_{1} + i_{2} + \cdots + i_{m}} \bsi  \ti{L} \quad \forall ~~ m \ge 1.
\end{equation}

Let us recall that a Maurer-Cartan (MC) element of $L$ is a degree $1$ element 
$\al \in L$ satisfying the equation
\begin{equation}
\label{int-MC}
\sum_{m=1}^{\infty} \frac{1}{m!} \cQ' \big( (\bsi \al)^m \big) = 0\,,
\end{equation}
where the infinite series in the left hand side makes sense due to 
conditions $L = \cF_1 L$, \eqref{int-L-complete}, and \eqref{int-Q-filtr-OK}.

We also recall that every $\infty$-morphism \eqref{int-U} of $L_{\infty}$-algebras 
gives us a map of sets 
$$
U_* : \MC(L) \to \MC(\ti{L})
$$
given by the formula
\begin{equation}
\label{int-U-star}
U_*(\al) : = \sum_{m=1}^{\infty} \frac{1}{m!} \bs\, U' \big( (\bsi \al)^m \big)\,,
\end{equation}
where $\MC(L)$ (resp. $\MC(\ti{L})$) denotes the set of MC elements of 
$L$ (resp. $\ti{L}$).

It is natural to view MC elements of $L$ as zero-cells of the simplicial set 
$\mMC_{\bul}(L)$ with 
\begin{equation}
\label{int-mMC}
\mMC_n(L) : = \MC(L \hotimes \Om_n)
\end{equation}
where $\Om_n$ the de Rham-Sullivan algebra of polynomial 
differential forms on the geometric simplex $\Delta^n$ with 
coefficients in $\bbk$ and 
$$
L \hotimes \Om_n  : = \varprojlim_{k} \big( (L/\cF_{k}L) \otimes \Om_n \big)\,.
$$
Due to \cite[Proposition 4.1]{EnhancedLie}, the simplicial set $\mMC_{\bul}(L)$
is a Kan complex (a.k.a. an $\infty$-groupoid).    

It is easy to see that the assignment $L  \mapsto \mMC_{\bul}(L)$ upgrades 
to a functor from the category of (filtered) $L_{\infty}$-algebras (with morphisms 
being $\infty$-morphisms) to the category of simplicial sets. 

The goal of this paper is to prove the following version\footnote{This theorem is already 
used in papers \cite{HAform} and \cite{Action}.} 
of the Goldman-Millson theorem:
\begin{thm}
\label{thm:main}
Let $L$ and $\ti{L}$ be filtered $L_{\infty}$-algebras and $U$ be an $\infty$-morphism
from $L$ to $\ti{L}$ compatible with the filtrations in the sense of \eqref{U-filtration-OK}.
If the linear term $\vf : L \to \ti{L}$ of $U$ gives us a quasi-isomorphism 
$$
\vf \big|_{\cF_m L} ~: ~\cF_m L ~ \to~  \cF_m \ti{L}
$$ 
for every $m \ge 1$ then
$$
\mMC_{\bul}(U) : \mMC_{\bul}(L) \to \mMC_{\bul}(\ti{L})
$$
is a weak equivalence of simplicial sets. 
\end{thm}
We would like to remark that Theorem \ref{thm:main} is a generalization 
of Proposition 4.9 from paper \cite{Ezra-infty} \footnote{We recently found a very nice proof of \cite[Proposition 4.9]{Ezra-infty} 
in \cite{BGNT}. See Proposition 3.4 in \cite{BGNT}.}.

Since it is more convenient to deal with shifted $L_{\infty}$-algebras 
(a.k.a. $\sLie$-algebras) than with usual $L_{\infty}$-algebras, the bulk of the presentation is 
given in the setting of  shifted $L_{\infty}$-algebras which are briefly reviewed in   
Section \ref{sec:prelim}. In this section, we also formulate the main result of this
paper (see Theorem \ref{thm:GM}), prove an important particular case of this theorem 
(see Proposition \ref{prop:Abelian}) and outline the structure of the proof of the main result.
In this respect, Section \ref{sec:prelim} may be viewed as more detailed introduction 
to our paper. 

The proof of the main result occupies Sections \ref{sec:pi0-level} and \ref{sec:higher}. 
In Section  \ref{sec:pi0-level}, we take care of the induced map from 
$\pi_0\big( \mMC_{\bul}(L) \big)$ to $\pi_0 \big( \mMC_{\bul}(\ti{L}) \big)$ and, 
in Section \ref{sec:higher}, we take care of the corresponding map 
for higher homotopy groups of the simplicial sets  $\mMC_{\bul}(L)$ 
and $\mMC_{\bul}(\ti{L})$. 

The three appendices at the end of the paper are devoted to proofs of various 
technical statements used in the bulk of the paper. In Appendix \ref{app:Ezra-lemma}, 
we recall the operators  $h^{i}_{n} \maps \Omega^{\bullet}_{n} \to \Omega^{\bullet-1}_{n}$
used in  \cite{Dupont}, \cite{Ezra-infty} and prove a version of Lemma 4.6 from 
\cite{Ezra-infty} for filtered $L_{\infty}$-algebras. In Appendix \ref{app:rectify}, we introduce the 
notion of rectified $1$-cell in  $\mMC_{\bul}(L)$ and prove that if two $0$-cells are connected 
in $\mMC_{\bul}(L)$ then they can be connected by a rectified $1$-cell. 
Finally, in Appendix \ref{app:con-ing}, we prove that, under certain conditions, an infinite 
sequence of $1$-cells in  $\mMC_{\bul}(L)$ can be ``composed''. We believe that this 
appendix has an independent value.

~\\

\noindent
\textbf{Acknowledgements:} Both authors acknowledge a partial support from NSF grant DMS-1161867.
C.L.R. also acknowledges support from the German Research Foundation (Deutsche
Forschungsgemeinschaft (DFG)) through the Institutional Strategy of the
University of G\"ottingen.

\subsubsection*{Notation and conventions}

A big part of our conventions is borrowed from our paper \cite{EnhancedLie}.
Thus, we work in the setting of unbounded cochain complexes of $\bbk$-vector 
spaces where $\bbk$ is any field of characteristic zero. We will frequently use the ubiquitous
abbreviation ``dg'' (differential graded) to refer to algebraic objects in the category 
of such cochain complexes. For a cochain complex $V$ we denote 
by $\bs V$ (resp. by $\bs^{-1} V$) the suspension (resp. the 
desuspension) of $V$\,. In other words, 
$$
\big(\bs V\big)^{\bul} = V^{\bul-1}\,,  \qquad
\big(\bs^{-1} V\big)^{\bul} = V^{\bul+1}\,. 
$$
The notation $\cZ^d(V)$ is reserved for the subspace of degree $d$ cocycles in $V$.

The notation $S_{n}$ is reserved for the symmetric group 
on $n$ letters and  $\Sh_{p_1, \dots, p_k}$ denotes 
the subset of $(p_1, \dots, p_k)$-shuffles 
in $S_n$, i.e.  $\Sh_{p_1, \dots, p_k}$ consists of 
elements $\si \in S_n$, $n= p_1 +p_2 + \dots + p_k$ such that 
$$
\begin{array}{c}
\si(1) <  \si(2) < \dots < \si(p_1),  \\[0.3cm]
\si(p_1+1) <  \si(p_1+2) < \dots < \si(p_1+p_2), \\[0.3cm]
\dots   \\[0.3cm]
\si(n-p_k+1) <  \si(n-p_k+2) < \dots < \si(n)\,.
\end{array}
$$

For a graded vector space (or a cochain complex) $V$
the notation $S(V)$ (resp. $\und{S}(V)$) is reserved for the
underlying vector space of the
symmetric algebra (resp. the truncated symmetric algebra) of $V$: 
$$
S(V) = \bbk \oplus V \oplus S^2(V) \oplus S^3(V) \oplus \dots\,, 
$$ 
$$
\und{S}(V) =  V \oplus S^2(V) \oplus S^3(V) \oplus \dots\,,
$$
where 
$$
S^n(V) = \big( V^{\otimes_{\bbk}\, n} \big)_{S_n}
$$
is the subspace of coinvariants with respect to the obvious action of $S_n$.

The graded vector space $\und{S}(V)$ is usually considered as the 
cocommutative coalgebra without counit and with the comultiplication $\D$ given by 
the formula 
$$
\D(v_1, v_2, \dots, v_n) : = \sum_{p = 1}^n 
\sum_{\si \in \Sh_{p, n-p} } 
(-1)^{\ve(\si; v_1, \dots, v_m)} (v_{\si(1)}, \dots, v_{\si(p)}) \otimes 
(v_{\si(p+1)}, \dots, v_{\si(n)})\,,
$$
where $(-1)^{\ve(\si; v_1, \dots, v_m)}$ is the Koszul sign factor
\begin{equation}
\label{ve-si-vvv}
(-1)^{\ve(\si; v_1, \dots, v_m)} := \prod_{(i < j)} (-1)^{|v_i | |v_j|}
\end{equation}
and the product in \eqref{ve-si-vvv} is taken over all inversions $(i < j)$ of $\si \in S_m$.

We often use the plain arrow $\to$ for $\infty$-morphisms of $L_{\infty}$-algebras 
(or shifted $L_{\infty}$-algebras).
Of course, it should be kept in mind that, in general, such morphisms are maps of 
the corresponding coalgebras but not the underlying cochain complexes.     

The abbreviation ``MC'' is reserved for the term ``Maurer-Cartan''.  

We denote by $\Om_{n}$ Sullivan's polynomial de Rham complex on the geometric 
$n$-simplex $\Delta^n$ with coefficients in $\bbk$ and recall that the collection 
$\{\Om_{n} \}_{n \geq 0}$ form a simplicial dg commutative $\bbk$-algebra. 
(See, for example, Section 3 in \cite{Ezra-infty}.) The notation $d$ is reserved 
for the de Rham differential on $\Om_n$. 

\section{Preliminaries}
\label{sec:prelim}

For technical reasons it is more convenient to deal with shifted $L_{\infty}$-algebras
(as in \cite{HAform} and \cite{EnhancedLie}) than with usual $L_{\infty}$-algebras. This is why we 
now  ``shift gears'' and pass to the setting of shifted $L_{\infty}$-algebras
for the rest of the paper. In this section, we briefly review  shifted $L_{\infty}$-algebras, 
introduce the operation $\curv$ (see \eqref{def-curv}) and formulate the main result of this
paper (Theorem \ref{thm:GM}).  We also prove an important 
particular case of the main theorem (Proposition \ref{prop:Abelian}) and outline 
the structure of the proof of Theorem \ref{thm:GM}.
 
Let us recall \cite{EnhancedLie} that a shifted $L_{\infty}$-algebra (a.k.a. $\sLie$-algebra)
is a cochain complex $(L, \pa)$ for which the cocommutative coalgebra
\begin{equation}
\label{coCom-L}
\und{S}(L)
\end{equation}
carries a degree $1$ coderivation $\cQ$ satisfying
\begin{equation}
\label{cQ-square}
\cQ^2 = 0 
\end{equation}
and the condition $\cQ(v) = \pa (v) ~~ \forall~~ v \in L$.  

Since every coderivation $\cQ$ of the coalgebra $\und{S}(L)$ is uniquely determined 
by its composition 
\begin{equation}
\label{cQ-pr}
\cQ' := p_{L} \circ \cQ : \und{S}(L) \to L
\end{equation}
with the projection $p_L : \und{S}(L) \to L$, an $\sLie$-algebra structure on a cochain complex 
$L$ is uniquely determined by the sequence of degree $1$ multi-brackets $(m \ge 2)$
\begin{equation}
\label{m-bracket}
\{~,~, \dots, ~\}_m : S^m(L) \to L
\end{equation}
$$
\{v_1, v_2, \dots, v_m\}_m = \cQ' ( v_1 v_2 \dots v_m)\,, \qquad v_j \in L\,.
$$

Furthermore, equation \eqref{cQ-square} is equivalent to the following sequence of relations
\begin{multline}
\label{sLie-relations}
\pa \{v_1, v_2, \dots, v_m\}_m + 
\sum_{i=1}^m (-1)^{|v_1| + \dots + |v_{i-1}|} 
\{v_1,  \dots, v_{i-1},  \pa v_i, v_{i+1}, \dots, v_{m}\}_m \\
+ \sum_{k=2}^{m-1} 
\sum_{\si \in \Sh_{k, m-k} } 
(-1)^{\ve(\si; v_1, \dots, v_m)}
\{ \{v_{\si(1)}, \dots, v_{\si(k)} \}_k , v_{\si(k+1)}, \dots, v_{\si(m)} \}_{m-k+1} = 0\,,
\end{multline}
where $(-1)^{\ve(\si; v_1, \dots, v_m)}$ is the Koszul sign factor defined in \eqref{ve-si-vvv}.

An $\infty$-morphism $U$ from a $\sLie$-algebra $(L, \cQ)$ to a $\sLie$-algebra 
$(\ti{L}, \ti{\cQ})$ is a homomorphism 
\begin{equation}
\label{U}
U : ( \und{S}(L), \cQ) \to (\und{S}(\ti{L}),  \ti{\cQ})   
\end{equation}
of the corresponding dg cocommutative coalgebras. 

We recall that any such coalgebra homomorphism $U$ is uniquely determined 
by its composition $U'$ with the projection $p_{\ti{L}}: \und{S}(\ti{L}) \to \ti{L}$: 
$$
U' : = p_{\ti{L}} \circ U : \und{S}(L) \to \ti{L}\,. 
$$  
% More precisely, given a degree zero map $U'  : \und{S}(L) \to \ti{L}$, the homomorphism 
% of coalgebras $U$ is restored by the formula  
% \begin{multline}
% \label{U-U-pr}
% U(v_1 v_2 \dots v_n) = 
%  \sum_{t \ge 1}
% \sum_{\substack{ k_1+ \dots + k_t = n \\[0.1cm] k_j \ge 1 }} 
% \sum_{\si \in \mSH_{k_1, k_2, \dots, k_t}}  
% U'(v_{\si(1)} \dots v_{\si(k_1)}) U'(v_{\si(k_1+1)} \dots v_{\si(k_1 + k_2)}) \dots \\
% \dots U'(v_{\si(n-k_t+1)} \dots v_{\si(n)})\,,
% \end{multline}
% where $\mSH_{k_1, k_2, \dots, k_t}$ 
% is  the subset of permutations in $\Sh_{k_1, k_2, \dots, k_t}$ satisfying 
% the condition 
% $$
% \si(1) < \si(k_1+1) <  \si(k_1+k_2+1) < \dots < \si(n-k_t+1)\,.
% $$

For every $\infty$-morphism \eqref{U} the map 
\begin{equation}
\label{vf}
\vf : = U' \Big|_{L} : L \to \ti{L}  
\end{equation}
is a morphism of cochain complexes and we call $\vf$ the linear term of $U$. 
An $\infty$-morphism $U$ for which $\vf$ induced an isomorphism 
$H^{\bul}(L) \to H^{\bul}(\ti{L})$ is called an $\infty$ quasi-isomorphism of 
$\sLie$-algebras.  

Following  \cite{EnhancedLie}, we call a $\sLie$-algebra $L$ {\it filtered} 
if the underlying cochain complex $(L, \pa)$ is equipped with a complete 
descending filtration,
\begin{equation}
\label{filtr-L}
L = \cF_{1}L \supset \cF_{2}L \supset  \cF_{3}L  \cdots
\end{equation}
\begin{equation}
\label{L-complete}
L =\varprojlim_{k} L/\cF_{k}L\,,
\end{equation}
which is compatible with the brackets, i.e.
\[
\Bigl \{\cF_{i_{1}}L,\cF_{i_{2}}L,\ldots,\cF_{i_{m}}L \Bigr \}_m \subseteq
\cF_{i_{1} + i_{2} + \cdots + i_{m}} L \quad \forall ~~ m >1.
\]

In this paper, we consider $\infty$-morphisms \eqref{U} of filtered 
$\sLie$-algebras which are compatible with the filtrations, i.e. 
\begin{equation}
\label{U-w-filtration}
U'( \cF_{i_1} L \otimes  \cF_{i_2} L \otimes \dots \otimes  \cF_{i_m} L) \subset \cF_{i_1 + i_2 + \dots + i_m} \ti{L}\,,
\end{equation}

Since $L = \cF_1 L$, condition \eqref{U-w-filtration} guarantees that, for every 
degree $0$ vector $\al \in L$, the infinite sum 
$$
\sum_{m \ge 1} \frac{1}{m!} U'( \al^m)
$$
is a well defined element of $\ti{L}$. We denote by $U_*$ the 
map of sets $L^0 \mapsto \ti{L}^0$ given by this formula  
\begin{equation}
\label{U-star}
U_*(\al) : = \sum_{m \ge 1} \frac{1}{m!} U'( \al^m)\,.
\end{equation}

We denote by $\curv$ the map of sets $L^0 \mapsto L^1$ 
given by the formula 
\begin{equation}
\label{def-curv}
\curv(\al) : = \pa \al + \sum_{m \ge 1} \frac{1}{m!} \{\al, \dots, \al\}_m\,. 
\end{equation}
For example, elements $\al \in L^0$ satisfying $\curv(\al) = 0$ are precisely 
MC elements of the $\sLie$-algebra $L$.
Various useful properties of the operation $\curv$ are listed in the
following proposition, whose proof is given in \cite{EnhancedLie}. 

\begin{prop}[Prop.\ 2.2 \cite{EnhancedLie}]
\label{prop:curv}
Let $L$ and $\ti{L}$ be filtered $\sLie$-algebras and $U$
be an $\infty$-morphism from $L$ to $\ti{L}$ 
satisfying \eqref{U-w-filtration}. Then for every $\al, \beta \in L^0,  v \in L$ 
we have 
\begin{equation}
\label{Bianchi}
\pa(\curv(\al)) + \sum_{m=1}^{\infty} \frac{1}{m!} \{\al, \dots, \al, \curv(\al)\}_{m+1} = 0\,,
\end{equation}
\begin{equation}
\label{U-star-curv}
\curv \big( U_*(\al) \big) = 
\sum_{m \ge 0} \frac{1}{m!} U' \big( \al^m  \curv(\al) \big)\,,
\end{equation}

\begin{equation}
\label{square-curv}
\pa^{\al} \circ \pa^{\al} (v) =  - \{\curv(\al), v\}^{\al}_2\,,
\end{equation}

\begin{equation}
\label{curv-sum}
\curv(\al + \beta) = \curv(\al) +  \pa^{\al}(\beta) + \sum_{m=2}^{\infty} \frac{1}{m!} \{\beta, \dots, \beta \}^{\al}_m\,,
\end{equation}
where $\displaystyle \pa^{\al} : = \pa + \sum_{m \ge 1}  \frac{1}{m!} \{\al, \dots, \al, \cdot \}_{m+1} $ and 
$\displaystyle \{\cdot, \ldots, \cdot \}^{\al}_n : =  \sum_{m \ge 1}  \frac{1}{m!} \{\al, \dots, \al, \cdot, \ldots, \cdot \}_{m+n} $\,.
\end{prop}
% \begin{proof}
% Equation \eqref{Bianchi} is proved in \cite[Lemma 4.5]{Ezra-infty}. 
% Equations \eqref{square-curv} and \eqref{curv-sum} follow from relations \eqref{sLie-relations}.

% To prove \eqref{U-star-curv} we use equation \eqref{U-U-pr} which implies that 
% \begin{equation}
% \label{U-cxp-al}
% U \big( \exp(\al)-1 \big) = \exp\big( U_*(\al) \big) -1\,, 
% \end{equation}
% where $\exp(\al)-1$ (resp. $ \exp\big( U_*(\al) \big) -1$) is considered 
% as the element of the completion of $\und{S}(L)$ (resp. $\und{S}(\ti{L})$)
% defined by the corresponding Taylor series. 

% A similar computation shows that, for every $\beta \in \ti{L}$, we have 
% \begin{equation}
% \label{ti-Q-cxp-beta}
% \ti{\cQ} \big( \exp(\beta) -1 \big) = \exp(\beta) \curv(\beta)\,,
% \end{equation}
% where $\ti{\cQ}$ is the coderivation\footnote{Here we tacitly assume that 
% $\cQ$ and $\ti{\cQ}$ are extended in the natural way to the completions of 
% $\und{S}(L)$ and $\und{S}(\ti{L})$, respectively.} 
% of $\und{S}(\ti{L})$ corresponding to the $\sLie$-algebra structure on $\ti{L}$. 

% On the other hand, 
% $$
% \ti{\cQ} \circ U = U \circ \cQ\,.
% $$

% Hence \eqref{U-cxp-al} and \eqref{ti-Q-cxp-beta} imply that
% \begin{equation}
% \label{curv-done}
% \ti{\cQ} \big( \exp \big( U_*(\al) \big) - 1 \big) =
% U \big( \exp(\al) \curv(\al) \big)\,.
% \end{equation}

% Applying the canonical projection $p_{\ti{L}}$ to both sides of 
% \eqref{curv-done} we get desired identity \eqref{U-star-curv}. 
% \end{proof}

\subsection{Twisting of $\sLie$-algebra structures by MC elements}
\label{sec:twisting}

Let us recall \cite[Section 2.4]{thesis}, \cite{DeligneTw}, \cite[Section 4]{Ezra-infty} that, given 
a MC element $\al$ of a filtered $\sLie$-algebra $L$ we can form a new filtered $\sLie$-algebra $L^{\al}$\,.
As a graded vector space with a filtration, $L^{\al} = L$; the differential $\pa^{\al}$ and the multi-brackets
$\{~,\dots, ~ \}^{\al}_m$ on $L^{\al}$ are defined by the formulas 
\begin{equation}
\label{diff-twisted}
\pa^{\al} (v) : = \pa (v)  + \sum_{k=1}^{\infty} \frac{1}{k!} \{\al, \dots, \al, v\}_{k+1}\,,
\end{equation}
\begin{equation}
\label{Linft_twisted}
\{v_1,v_2,\cdots, v_m \}^{\alpha}_m : =
\sum_{k=0}^{\infty} \frac{1}{k!} \{\al, \dots, \al, 
v_1, v_2,\cdots, v_m \}_{k+m}\,.
\end{equation}

Equation \eqref{U-star-curv} implies that, for every $\infty$-morphism $U$ \eqref{U}
compatible with the filtrations and for every MC element $\al \in L$,  
$$
U_*(\al)
$$
is a MC element of $\ti{L}$. Finally,  for every $\infty$-morphism $U$ \eqref{U}
compatible with the filtrations and for every MC element $\al \in L$,
we can construct a new $\infty$-morphism
\[
U^{\alpha} \maps L^{\alpha} \to \ti{L} ^{U_*(\alpha)}
\]
with 
\begin{equation} 
\label{twisted_map}
p_{\ti{L}} \circ U^{\alpha} (v_1 v_2 \dots v_m) : = \sum_{k=0}^{\infty} \frac{1}{k!}\,  U' (\al^k \, v_1 v_2 \dots v_m)\,.
\end{equation}

\subsection{The Deligne-Getzler-Hinich (DGH) $\infty$-groupoid and Theorem \ref{thm:GM}}
\label{sec:DGH-and-thm}

For every filtered $\sLie$-algebra $L$ we introduce the following collection of 
filtered $\sLie$-algebras 
$$
L \hotimes \Om_n\,, \qquad n \ge 0
$$
where $L$ is considered with the topology coming from the filtration and the dg commutative 
algebra $\Om_n$ is considered with the discrete topology. 

The simplicial set $\mMC_{\bul}(L)$ with 
$$
\mMC_n(L) : = \MC (L \hotimes \Om_n) 
$$
is the main hero of this paper.  Due to \cite[Proposition 4.1]{EnhancedLie}, 
$\mMC_{\bul}(L)$ is a Kan complex (a.k.a. an $\infty$-groupoid) for every 
filtered $\sLie$-algebra $L$. We call this simplicial set the Deligne-Getzler-Hinich (DGH)
$\infty$-groupoid of $L$.

For example, $0$-cells of  $\mMC_{\bul}(L)$ are precisely MC elements of $L$ and 
$1$-cells in $\mMC_{\bul}(L)$ are elements 
$$
\beta = \beta_0(t_0) + d t_0 \beta_1(t_0)\,, \qquad 
\beta_0(t_0)\in L^0 \hotimes \bbk[t_0]\,, \qquad 
\beta_1(t_0)\in L^{-1} \hotimes \bbk[t_0]\,,
$$
satisfying the equations 
\begin{equation}
\label{curv-beta-0}
\curv\big( \beta_0(t_0) \big) =0
\end{equation}
and 
\begin{equation}
\label{diff-beta-0}
\frac{d}{d t_0} \beta_0(t_0)  =  \pa^{\beta_0(t_0)} \, \beta_1(t_0) \,,
\end{equation}
where $\pa^{\beta_0(t_0)}$ denotes the differential on $L\hotimes \bbk[t_0]$ twisted by
the MC element $\beta_0(t_0)$ (as in Sec. \ref{sec:twisting}).
The zeroth face $\md_0 (\beta) $ (resp. the first face $\md_1(\beta)$) of $\beta$ is the 
MC element $\beta_0(0)$ (resp. $\beta_0(1)$) of $L$.

Let us recall that any $\infty$-morphism $U$ \eqref{U}
of filtered $\sLie$-algebras satisfying \eqref{U-w-filtration}
gives us the collection of $\infty$-morphisms of $\sLie$-algebras
\begin{equation}
\label{U-level-n}
\begin{split}
U^{(n)} \maps  L \hotimes \Om_n &~ \to ~ \ti{L} \hotimes \Om_n \\
U^{(n)} \bigl(v_1 \tensor \omega_1, v_2 \tensor \omega_2,\ldots,v_m \tensor
\omega_m \bigr) &= \pm U(v_1,v_2,\ldots,v_m)\tensor \omega_1 \omega_2 \cdots \omega_m,
\end{split}
\end{equation}
where $v_i \in L$, $\omega_i \in \Omega_n$, and $\pm$ is the usual
Koszul sign. This collection is obviously compatible with 
all the faces and all the degeneracies.  Hence, $U$ induces a morphism of simplicial sets 
\begin{equation}
\label{MC-U}
\mMC_{\bul}(U) : \mMC_{\bul}(L) \to  \mMC_{\bul}(\ti{L})
\end{equation}
given by the formula 
\begin{equation}
\label{MC-U-formula}
\mMC_{n}(U) (\al) : = U^{(n)}_*(\al)\,.
\end{equation}

The main result of this paper is the following version of the 
Goldman-Millson theorem \cite{GM}:
\begin{thm}
\label{thm:GM}
Let $L$ and $\ti{L}$ be filtered $\sLie$-algebras and 
$U : L \to \ti{L}$ be an $\infty$-morphism compatible with 
the filtrations on $L$ and $\ti{L}$. If the linear term 
$\vf$ of $U$ induces a quasi-isomorphism of cochain complexes 
\begin{equation}
\label{vf-cF-n}
\vf \Big|_{\cF_n L } :
\cF_n L \to \cF_n \ti{L}  
\end{equation}
for every $n$, then $\mMC_{\bul}(U)$ is a homotopy equivalence of simplicial sets. 
\end{thm}
\begin{remark}
\label{rem:thms-r-equiv}
Due to the obvious equivalence between the category of filtered $\sLie$-algebras 
and the category of filtered $L_{\infty}$-algebras, Theorem \ref{thm:main} stated in the 
Introduction is equivalent to Theorem \ref{thm:GM}.
\end{remark}

\subsection{What if we deal with abelian $\sLie$-algebras?}
If $\{~,\dots, ~\}_m = 0$ for all $m \ge 2$, the $\sLie$-algebra $L$ is nothing but 
an unbounded cochain complex. In this case we may consider 
the simplicial set 
$$
\mMC_n(L) : = \MC(L\otimes \Om_n)
$$ 
with usual tensor product $\otimes$ which is equivalent to 
considering $L$ with the silly filtrations: 
$$
L = \cF_1 L \supset \cF_2 L =  \cF_3 L = \dots = \bfzero\,.
$$ 

Since all the brackets are zero, the set 
$\MC(L\otimes \Om_n)$ is precisely the subspace 
of degree zero cocycle in $L\otimes \Om_n$:
\begin{equation}
\label{MC-easy}
\MC(L\otimes \Om_n) = \cZ^0 \big( L\otimes \Om_n,   \pa + d  \big)\,.
\end{equation}
In particular, the simplicial set $\mMC_{\bul}(L)$ is a simplicial vector space. 
 
Let us consider the very special case of Theorem \ref{thm:GM} in which 
both $L$ and $\ti{L}$ are abelian (with the above silly filtrations)
and $U$ is a strict morphism $\vf: L \to \ti{L}$, i.e. 
$$
U' (v_1 v_2 \dots v_m) = 0 \quad \textrm{if} \quad m \ge 2\,.
$$ 

In this case we have 
\begin{prop}
\label{prop:Abelian}
If $\vf : L \to \ti{L}$ is a quasi-isomorphism of cochain complex and 
$L$, $\ti{L}$ are viewed as the abelian $\sLie$-algebras then 
the induced map  
\begin{equation}
\label{mMC-vf}
\mMC_{\bul}(\vf) : \mMC_{\bul}(L) \to \mMC_{\bul}(\ti{L})
\end{equation}
is a weak equivalence of simplicial vector spaces. 
\end{prop}
\begin{remark}
\label{rem:prop-Abelian}
In principle, Proposition \ref{prop:Abelian} is consequence of  \cite[Proposition 5.1]{Ezra-infty} 
and \cite[Corollary 5.11]{Ezra-infty}. The proof given below is essentially a more detailed presentation 
of an argument from the proof of \cite[Corollary 5.11]{Ezra-infty}.
\end{remark}
\begin{proof}[ of Proposition \ref{prop:Abelian}]
First, we recall\footnote{See, for example, Corollary 2.5 in \cite[Sec. III.2]{Goerss-Jardine} } that 
for every simplicial vector space $V_{\bul}$ 
\begin{equation}
\label{homot-groups-V}
\pi_0(V_{\bul}) = H_0(\cM(V_{\bul}))\,, \qquad
\pi_i(V_{\bul}, 0) = H_i(\cM(V_{\bul}))\,, \qquad \forall ~~ i \ge 1\,,
\end{equation}
where $\cM(V_{\bul})$ is the Moore chain complex of the simplicial vector space $V_{\bul}$: 
\begin{equation}
\label{Moore-V}
\cM(V_{\bul}) : = \bigoplus_{n} \bs^{n} V_n
\end{equation}
with the differential 
$$
\md : = \sum_{i=0}^n (-1)^i \md_i  ~:~ \bs^{n} V_n \to  \bs^{n-1} V_{n-1}\,.
$$
Furthermore, $\pi_i(V_{\bul}, v)$ is naturally isomorphic to $\pi_i(V_{\bul}, 0)$ for every 
base point $v \in V_0$\,.

Therefore, a map $f: V_{\bul} \to W_{\bul}$ of simplicial vector spaces is a weak equivalence if 
and only if the induced map 
$$
\cM(f) : \cM(V_{\bul}) \to \cM(W_{\bul})
$$
is a quasi-isomorphism of chain complexes. 

Using the elementary forms introduced on page 282 of \cite[Section 3]{Ezra-infty}, one 
can identify the simplicial cochain complex 
\begin{equation}
\label{C-n}
C_n : = C^{\bul}_{simpl} (\D^n, \bbk)
\end{equation}
with a subcomplex of $\Om_n$\,.

Eq. (3-6) from \cite[Section 3]{Ezra-infty} defines projections 
\begin{equation}
\label{P-n}
P_n : \Om_n \to C_n 
\end{equation}
which assemble into a map of simplicial cochain complexes. 

Due to \cite[Theorem 3.7]{Ezra-infty} there exists a simplicial endomorphism 
\begin{equation}
\label{ms}
\ms_n : \Om^{\bul}_n \to  \Om^{\bul - 1}_n 
\end{equation}
which satisfies
\begin{equation}
\label{ms-homotopy}
\id - P_n = d \ms_n + \ms_n d
\end{equation}
and 
\begin{equation}
\label{P-ms}
P_n \circ \ms_n = 0\,, \qquad \forall~~ n \ge 0\,.
\end{equation} 
We call $\ms_{\bul}$ the {\it Dupont operator}\footnote{A very similar operator was introduced 
in Dupont's proof \cite[Eq. (2.25), page 25]{Dupont} of the de Rham theorem.}.

We now extend the operators $P_n$, $\ms_n$, in the obvious way, to 
$$
L \otimes  \Om_n 
$$
for any cochain complex $(L, \pa)$. Then, equation \eqref{ms-homotopy} becomes 
\begin{equation}
\label{ms-homotopy-L}
\id - P_n = (\pa + d) \ms_n + \ms_n (\pa + d)\,.
\end{equation}

Equations \eqref{P-ms} and \eqref{ms-homotopy-L} imply that 
the projection $P_{\bul}$ gives us the following short exact sequence 
of simplicial vector spaces 
\begin{equation}
\label{ses-P-bul}
\begin{tikzpicture}
\matrix (m) [matrix of math nodes, row sep=1em, column sep=2em]
{ \bfzero &  (\pa + d)  \ms_{\bul} \big( (L \otimes  \Om_{\bul} )^0 \big)  & \cZ^0(L \otimes  \Om_{\bul} ) 
& \cZ^0(L \otimes  C_{\bul}) & \bfzero \\ };
\path[->,font=\scriptsize]
(m-1-1) edge (m-1-2)  (m-1-2) edge (m-1-3) (m-1-3) edge node[above] {$P_{\bul}$} (m-1-4)  (m-1-4) edge  (m-1-5); 
\end{tikzpicture}
\end{equation}

Using equations \eqref{P-ms} and \eqref{ms-homotopy-L} again, we get 
$$
((\pa + d)  \ms_{\bul})^2 =  (\pa + d)  \ms_{\bul}   (\pa + d) \ms_{\bul} = 
  (\pa + d)  \big( \ms_{\bul}   (\pa + d)  + (\pa +d) \ms_{\bul} \big) \ms_{\bul}
$$
$$
= (\pa + d) \big(\id - P_{\bul} \big) \ms_{\bul} =  (\pa + d)  \ms_{\bul} 
$$
which means that the operator $(\pa + d)  \ms_{\bul} $ is a retract of simplicial vector spaces
\begin{equation}
\label{L-Omega-0}
(L \otimes  \Om_{\bul} )^0 \xto{(\pa + d)  \ms_{\bul}} (\pa + d)  \ms_{\bul} \big( (L \otimes  \Om_{\bul} )^0 \big)\,.
\end{equation}

On the other hand, Lemma 3.2 from \cite{Ezra-infty} implies that 
$(L \otimes  \Om_{\bul} )^0$
% simplicial 
% vector space \eqref{L-Omega-0} 
has an acyclic Moore complex. Therefore, 
the Moore complex 
$$
\cM \Big(\, (\pa + d)  \ms_{\bul} \big( (L \otimes  \Om_{\bul} )^0 \big) \, \Big)
$$
of the simplicial vector space $ (\pa + d)  \ms_{\bul} \big( (L \otimes  \Om_{\bul} )^0 \big)$
is also acyclic. 

Thus we conclude that the projection $P_{\bul}$ gives us a quasi-isomorphism of 
chain complexes: 
\begin{equation}
\label{cM-P-bul}
\cM(P_{\bul}) : \cM\big( \cZ^0(L \otimes  \Om_{\bul} ) \big) \to
 \cM\big( \cZ^0(L \otimes C_{\bul} ) \big)\,.
\end{equation}

Let us now observe that the simplicial vector space 
$$
\cZ^0(L \otimes C_{\bul} )
$$
is precisely the result of applying the Dold-Kan functor to the chain complex 
which is obtained from the  truncation
\begin{equation}
\label{truncate-L}
\begin{tikzpicture}
\matrix (m) [matrix of math nodes, row sep=1em, column sep=2em]
{ \stackrel{~}{\cdots} & L^{-2} & L^{-1} & \cZ^0(L) \\ };
\path[->,font=\scriptsize]
(m-1-1) edge node[above] {$\pa$} (m-1-2)  (m-1-2) edge  node[above] {$\pa$} (m-1-3) 
(m-1-3) edge node[above] {$\pa$} (m-1-4); 
\end{tikzpicture}
\end{equation}
of $(L, \pa)$ by reversing the grading. 

Since $\vf : L \to \ti{L}$ is a quasi-isomorphism of cochain complexes, it also 
induces a quasi-isomorphism of the corresponding truncations of $L$ and $\ti{L}$.

Thus we get a commutative diagram of chain complexes 
\begin{equation}
\label{Moore-diag}
\begin{tikzpicture}
\matrix (m) [matrix of math nodes, row sep=2em, column sep=3em]
{ \cM\big( \cZ^0(L \otimes  \Om_{\bul} ) \big) &  \cM \big( \cZ^0(L \otimes C_{\bul} ) \big) \\
 \cM\big( \cZ^0(\ti{L} \otimes  \Om_{\bul} ) \big) &  \cM \big( \cZ^0(\ti{L} \otimes C_{\bul} ) \big) \\ };
\path[->,font=\scriptsize]
(m-1-1) edge node[above] {$\cM(P_{\bul})$} (m-1-2)  edge node[left] {$\psi$} (m-2-1)
(m-1-2)  edge node[auto] {$\vf_*$} (m-2-2)
(m-2-1)  edge node[above] {$\cM(P_{\bul})$} (m-2-2); 
\end{tikzpicture}
\end{equation}
where $\psi : = \cM\circ \mMC_{\bul}(\vf)$\,.

We already proved that the horizontal arrows in \eqref{Moore-diag} are quasi-isomorphisms. 
Furthermore, the right vertical arrow is a quasi-isomorphism by the Dold-Kan correspondence. 

Thus $\cM\circ \mMC_{\bul}(\vf)$ is also a quasi-isomorphism of chain complexes
and Proposition \ref{prop:Abelian} follows. 
\end{proof}

\subsection{The structure of the proof of Theorem \ref{thm:GM}}
\label{sec:str-proof}

By the definition of weak equivalence of simplicial sets, we have to show that 
\begin{itemize}

\item the map  $\mMC_{\bul}(U)$ induces a bijection
\begin{equation}
\label{pi-0-level}
\pi_0\big( \mMC_{\bul}(L) \big) \to \pi_0\big( \mMC_{\bul}(\ti{L}) \big)
\end{equation}

\item and for every $0$-cell $\al$ in $\mMC_{\bul}(L)$ the map 
$\mMC_{\bul}(U)$ induces group isomorphisms
\begin{equation}
\label{pi-i-level}
\pi_i \big( \mMC_{\bul}(L), \al \big)~ \to ~ \pi_i \big( \mMC_{\bul}(\ti{L}), U_*(\al) \big)\,, 
\qquad \forall~~ i \ge 1\,.
\end{equation}
\end{itemize}

In Section \ref{sec:pi0-level}, we give a detailed proof of the fact that 
map \eqref{pi-0-level} is a bijection of sets. In Section \ref{sec:higher}, 
we apply Proposition \ref{prop:Abelian} to the map of simplicial sets 
$$
\mMC_{\bul}(L / \cF_2 L)  ~\to~   \mMC_{\bul}(\ti{L} / \cF_2 \ti{L})
$$
and then prove, by induction on $n$, that $\mMC_{\bul}(U)$ induces an isomorphism 
of groups   
$$
\pi_i \big( \mMC_{\bul}(L / \cF_n L), 0 \big)~ \to ~ \pi_i \big( \mMC_{\bul}(\ti{L}/ \cF_n \ti{L} ), 0 \big)
$$
for every $i\ge 1$ and $n \ge 2$\,.

Then technical Lemma \ref{lem:zero-base-pt} implies that the corresponding map 
$$
\pi_i \big( \mMC_{\bul}(L / \cF_n L), \al \big)~ \to ~ \pi_i \big( \mMC_{\bul}(\ti{L}/ \cF_n \ti{L} ), U_*(\al) \big)
$$ 
is an isomorphism of groups for every $\al \in \MC(L)$, $i \ge 1$, and $n \ge 2$. 

Combining this statement with the results of Section \ref{sec:pi0-level} applied to 
the corresponding $\infty$-morphism of quotients 
$$
L / \cF_n L ~\to~ \ti{L} / \cF_n \ti{L}
$$
we conclude that $\mMC_{\bul}(U)$ induces a weak equivalence of simplicial sets
$$
\mMC_{\bul}(L / \cF_n L) ~\to~ \mMC_{\bul} (\ti{L} / \cF_n \ti{L})
$$
for all $n \ge 2$. 

Finally, we deduce Theorem \ref{thm:GM} from standard facts about 
maps of towers of simplicial sets from \cite[Section VI]{Goerss-Jardine}. 

Let us remark that the explicit constructions given in the proof of injectivity 
and in the proof of surjectivity of \eqref{pi-0-level} have an independent value. 
These constructions may be used in establishing various properties of
homotopy algebraic structures or $\infty$-morphisms of homotopy algebras 
\cite{HAform}, \cite{notes}, \cite{LV} for which only the existence statements are proved.

\section{$\mMC_{\bul}(U)$ induces a bijection on the level of $\pi_0$}
\label{sec:pi0-level}

We start the proof of the fact that \eqref{pi-0-level} is a bijection 
with two obvious observations. 
Since the linear term $\vf$ of $U$ induces quasi-isomorphism \eqref{vf-cF-n}
for every $n$, $\vf$ induces quasi-isomorphisms 
\begin{equation}
\label{quotient-n}
L \big/ \cF_n L ~\to~ \ti{L} \big/ \cF_n \ti{L}  
\end{equation}
\begin{equation}
\label{quotient-n1}
\cF_n L \big/ \cF_{n+1} L ~\to~ \cF_n \ti{L} \big/ \cF_{n+1} \ti{L}  
\end{equation}
for every $n$.

\subsection{Map \eqref{pi-0-level} is injective}
\label{sec:injective}
To prove that map \eqref{pi-0-level}  induced by $U_*$ is injective, we denote by 
$\al$ and $\al'$ MC elements of $L$ for which $U_*(\al)$ and $U_*(\al')$ 
are connected by a $1$-cell in  $\mMC_{\bul}(\ti{L})$. 
From this, we will construct a sequence of MC elements in $L$ 
\begin{equation}
\label{al-n-seq}
\{\alpha^{(n)}\}_{n \geq 1}\,, 
\end{equation}
a sequence of rectified $1$-cells 
\begin{equation}
\label{rho-n-seq}
\big\{   \rho^{(n)}  = \rho^{(n)}_0(t_0) + d t_0  \rho^{(n)}_1   \big\}_{n \geq 1}
\end{equation}
in $\mMC_{\bul}(L)$ and a sequence of rectified $1$-cells 
\begin{equation}
\label{ti-beta-n-seq}
\big\{  \ti{\beta}^{(n)}  = \ti{\beta}^{(n)}_0(t_0) + d t_0  \ti{\beta}^{(n)}_1     \big\}_{n \geq 1}
\end{equation}
in $\mMC_{\bul}(\ti{L})$ satisfying the following properties: 
\begin{equation}
\label{al-n-eq-1}
\al^{(1)} = \al\,, 
\end{equation}
\begin{equation}
\label{al-n-to-al-pr}
\al' - \al^{(n)} \in \cF_n L\,, 
\end{equation}
\begin{equation}
\label{rho-n-1}
\rho^{(n)}_1 \in  \cF_{n} L\,, 
\end{equation}
\begin{equation}
\label{rho-n-begin-end}
\rho^{(n)}_0(0) = \al^{(n)}\,, \qquad 
\rho^{(n)}_0(1) = \al^{(n+1)}\,,
\end{equation}
\begin{equation}
\label{ti-beta-n-1}
\ti{\beta}^{(n)}_1 \in  \cF_{n} \ti{L}\,, 
\end{equation}
and 
\begin{equation}
\label{ti-beta-n-begin-end}
\ti{\beta}^{(n)}_0(0) = U_*(\al^{(n)})\,, \qquad 
\ti{\beta}^{(n)}_0(1) = U_*(\al')\,.
\end{equation}

Then, by concatenating together the infinite sequence of $1$-cells  $\rho^{(n)}$, we will obtain a
1-cell in $\mMC_{\bul}(L)$ connecting $\alpha$ with $\alpha'$.

\subsubsection{The inductive construction of sequences \eqref{al-n-seq}, \eqref{rho-n-seq}, and \eqref{ti-beta-n-seq}}

Since  $U_*(\al)$ and $U_*(\al')$ are connected by a $1$-cell in  $\mMC_{\bul}(\ti{L})$,
Lemma \ref{lem:rectify} from Appendix \ref{app:rectify} implies that there exists 
a rectified $1$-cell 
\begin{equation}
\label{ti-beta-1}
 \ti{\beta}^{(1)}  = \ti{\beta}^{(1)}_0(t_0) + d t_0  \ti{\beta}^{(1)}_1 
\end{equation}
such that 
\begin{equation}
\label{ti-beta-1-0}
\ti{\beta}^{(1)}_0(0) = U_*(\al)\,,
\end{equation}
and
\begin{equation}
\label{ti-beta-1-1}
\ti{\beta}^{(1)}_0(1) = U_*(\al')\,.
\end{equation}
The desired condition  $\ti{\beta}^{(1)}_1 \in \cF_1 \ti{L}$ is satisfied 
automatically since $\ti{L} =  \cF_1 \ti{L}$. 

Thus we take $\al^{(1)} := \al$ as the base of our induction and assume that 
we already constructed MC elements 
$$
\al^{(1)}, \al^{(2)}, \dots, \al^{(n)} \,, 
$$ 
$1$-cells 
$$
\rho^{(1)}, \rho^{(2)}, \dots, \rho^{(n-1)}  
$$
in $\mMC_{\bul}(L)$ and $1$-cells 
$$
\ti{\beta}^{(1)}, \ti{\beta}^{(2)}, \dots, \ti{\beta}^{(n)}  
$$
in $\mMC_{\bul}(\ti{L})$ satisfying all the desired properties. 

Our goal is to construct a rectified $1$-cell 
\begin{equation}
\label{rho-n}
\rho^{(n)}  = \rho^{(n)}_0(t_0) + d t_0  \rho^{(n)}_1
\end{equation}
in $\mMC_{\bul}(L)$ and a rectified $1$-cell
\begin{equation}
\label{ti-beta-n1}
\ti{\beta}^{(n+1)}  = \ti{\beta}^{(n+1)}_0(t_0) + d t_0  \ti{\beta}^{(n+1)}_1
\end{equation}
such that $\rho^{(n)}_1 \in \cF_n L$, $\ti{\beta}^{(n+1)}_1 \in  \cF_{n+1} \ti{L}$,
\begin{equation}
\label{rho-n-0}
\rho^{(n)}_0(0) = \al^{(n)}\,,
\end{equation}
\begin{equation}
\label{the-end-better}
\al' -  \rho^{(n)}_0(1) \in  \cF_{n+1} L\,,
\end{equation}
\begin{equation}
\label{ti-beta-n1-0}
\ti{\beta}^{(n+1)}_0(0) = U_*\big( \rho^{(n)}_0(1) \big)
\end{equation}
and 
\begin{equation}
\label{ti-beta-n1-1}
\ti{\beta}^{(n+1)}_0(1) = U_*\big( \al'  \big)\,.
\end{equation}
Then setting $\al^{(n+1)} : =  \rho^{(n)}_0(1)$ would complete the inductive step. 

Diagram \eqref{diag-induction} below shows all the links between $0$-cells and $1$-cells
under consideration.
The elements which should be constructed in the inductive step are shown in blue. 
\begin{equation}
\label{diag-induction}
\begin{tikzpicture}
\matrix (m) [matrix of math nodes, row sep=3em, column sep=1em]
{ \al = \al^{(1)} &  \al^{(2)}   &  \al^{(3)} &\cdots &  \al^{(n-1)} &  \al^{(n)}  & \textcolor{blue}{ \al^{(n+1)} }   \\
  U_* (\al^{(1)}) &   ~ & U_* (\al^{(2)})  &  \cdots &  U_*(\al^{(n-1)}) & U_*( \al^{(n)})  & U_*(\textcolor{blue}{ \al^{(n+1)} })  \\ 
 ~ &~  & ~ & U_*(\al')\,. & ~ & ~ & ~ \\  };
\path[->, font=\scriptsize]
(m-1-1) edge  node[above] {$\rho^{(1)}$} (m-1-2)   (m-1-2)  edge node[above] {$\rho^{(2)}$}
(m-1-3)   (m-1-5)  edge node[above] {$\rho^{(n-1)}$}
(m-1-6)   (m-1-6)  edge node[above] {$\textcolor{blue}{ \rho^{(n)} }$} (m-1-7)
(m-2-1)  edge  node[above] {$~~\ti{\beta}^{(1)}$} (m-3-4) 
(m-2-3)  edge  node[above] {$~~~~\ti{\beta}^{(2)}$} (m-3-4)
(m-2-5)  edge  node[above] {$\ti{\beta}^{(n-1)}~~~~$} (m-3-4)
(m-2-6)  edge  node[above] {$\ti{\beta}^{(n)}~$} (m-3-4)
(m-2-7)  edge  node[below] {$\textcolor{blue} {~~\ti{\beta}^{(n+1)}}$} (m-3-4);
\end{tikzpicture}
\end{equation}

Since both $\al^{(n)}$ and $\al'$ satisfy the MC equation
and $\al' - \al^{(n)} \in \cF_n L$\,, equation \eqref{curv-sum} implies that 
the difference 
\begin{equation}
\label{al-pr-minus-al-n}
\al' - \al^{(n)}
\end{equation}
represents a cocycle in the quotient complex 
$$
\cF_n L / \cF_{n+1} L\,.
$$

Let us prove that 
\begin{claim}
\label{cl:diffenerence-exact}
The difference in \eqref{al-pr-minus-al-n} represents 
a coboundary in $\cF_n L / \cF_{n+1} L$\,. 
\end{claim} 
\begin{proof}
Let us consider the $1$-cell $\ti{\beta}^{(n)}$ in $\mMC_{\bul}(\ti{L})$ which connects $U_*(\al^{(n)})$ 
to $U_*(\al')$\,.
  
The MC equation for $\ti{\beta}^{(n)}$ is equivalent to the pair of equations
\begin{equation}
\label{MC-beta-0}
\curv( \ti{\beta}^{(n)}_0(t)) = 0
\end{equation}
and 
\begin{equation}
\label{beta-0-diff-eq}
\frac{d}{d t} \ti{\beta}^{(n)}_0(t) = \pa^{\ti{\beta}^{(n)}_0(t)} \ti{\beta}^{(n)}_1\,,
\end{equation}
where $\pa^{\ti{\beta}^{(n)}_0(t)}$ denotes the differential twisted by $\ti{\beta}^{(n)}_0(t)$
(see Section \ref{sec:twisting}).

Integrating both sides of \eqref{beta-0-diff-eq}, we deduce that
\begin{multline}
\label{U-al-pr-U-al-n}
U_*(\al') - U_*(\al^{(n)}) = \int_0^1  d t\,  \pa^{\ti{\beta}^{(n)}_0(t)} \ti{\beta}^{(n)}_1  \\
 = \pa \ti{\beta}^{(n)}_1   + \sum_{m \ge 1} \frac{1}{m!} \int_0^1  d t \{\ti{\beta}^{(n)}_0(t), \dots,  \ti{\beta}^{(n)}_0(t),   \ti{\beta}^{(n)}_1\}_{m+1}\,. 
\end{multline}
Thus the difference $U_*(\al') - U_*(\al^{(n)})$ represents a coboundary in the 
quotient 
$$
\cF_n \ti{L} / \cF_{n+1} \ti{L} \,.
$$

On the other hand, 
\begin{equation}
\label{vf-al-pr-vf-al-n}
U_*(\al') - U_*(\al^{(n)})  - \vf(\al' - \al^{(n)}) \in \cF_{n+1} \ti{L}\,. 
\end{equation}
Therefore, since $\vf$ induces quasi-isomorphism 
\eqref{quotient-n1},  $\al' - \al^{(n)}$ indeed 
represents a coboundary in $\cF_n L / \cF_{n+1} L$\,.
\end{proof}

Claim \ref{cl:diffenerence-exact} implies that there exists 
\begin{equation}
\label{n-rho-1}
\rho^{(n)}_1 \in \cF_n L 
\end{equation}
such that
\begin{equation}
\label{al-pr-al-rho}
\al'- \al^{(n)} - \pa \rho^{(n)}_1 \in \cF_{n+1} L\,. 
\end{equation}

Let us denote by $\rho^{(n)}$ the $1$-cell
\begin{equation}
\label{rho-n-1-cell}
\rho^{(n)} = \rho^{(n)}_0 (t) + d t \rho^{(n)}_1\,, 
\end{equation}
where $\rho^{(n)}_0(t)$ is the unique solution to this integral equation: 
\begin{equation}
\label{rho-n-0-t}
\rho^{(n)}_0(t) = \al^{(n)} + \int_0^t d t_1 \pa^{\rho^{(n)}_0(t_1)} (\rho^{(n)}_1)\,.
\end{equation}

It is easy to see that the new MC element 
\begin{equation}
\label{al-n1}
\al^{(n+1)} : = \rho^{(n)}_0 (1)
\end{equation}
satisfies the condition 
$$
\al^{(n+1)} - \al^{(n)} - \pa \rho_1 \in \cF_{n+1} L
$$
and hence 
\begin{equation}
\label{al-pr-minus-al-n1}
\al' - \al^{(n+1)} \in  \cF_{n+1} L\,.
\end{equation}

We claim that 
\begin{claim}
\label{cl:rho-1-choice}
The element $\rho^{(n)}_1$ in \eqref{n-rho-1} can be chosen 
in such a way that 
\begin{equation}
\label{tibeta1-vf-rho1}
\ti{\beta}^{(n)}_1 - \vf(\rho^{(n)}_1) - \pa (\ti{\ga}) \in \cF_{n+1} \ti{L}
\end{equation}
for some $\ti{\ga} \in \cF_n \ti{L}$\,. 
Here  $\ti{\beta}^{(n)}_1$ is the degree $-1$ component of 
the $1$-cell $\ti{\beta}^{(n)}$ which connects $U_*(\al^{(n)})$ to 
$U_*(\al')$\,.
\end{claim}
\begin{proof}
Indeed, due to \eqref{U-al-pr-U-al-n} and inclusion 
\eqref{vf-al-pr-vf-al-n}, we have 
$$
\vf(\al') - \vf(\al^{(n)}) - \pa  \ti{\beta}^{(n)}_1 \in  \cF_{n+1} \ti{L} \,.
$$
 
Combining this inclusion with \eqref{al-pr-al-rho}, we conclude 
that 
$$
\pa(\ti{\beta}^{(n)}_1 -  \vf (\rho^{(n)}_1)) \in \cF_{n+1} \ti{L}. 
$$

Therefore, since $\vf$ induces quasi-isomorphism \eqref{quotient-n1},
there exists $\rho' \in \cF_n L^{-1}$ and $\ti{\ga} \in  \cF_n \ti{L}^{-2}$ such that
\begin{equation}
\label{diff-rho-pr-1}
\pa(\rho') \in \cF_{n+1} L
\end{equation}
and 
\begin{equation}
\label{ti-beta1-rho1-rho-pr}
\ti{\beta}^{(n)}_1 -  \vf (\rho^{(n)}_1) - \vf(\rho') - \pa (\ti{\ga}) \in \cF_{n+1} \ti{L}\,.
\end{equation}

Inclusion \eqref{diff-rho-pr-1} implies that we can safely replace 
$\rho^{(n)}_1$ in \eqref{n-rho-1} without violating inclusion \eqref{al-pr-al-rho}.

Claim \eqref{cl:rho-1-choice} is proved.
\end{proof}

Thus we constructed the desired ``blue'' elements $\rho^{(n)}$ and 
$\al^{(n+1)}$ in diagram \eqref{diag-induction}.

Let us now consider the $1$-cell
\begin{equation}
\label{U-star-rho-n}
\ti{\rho}^{(n)}  = \ti{\rho}^{(n)}_0(t) + d t \ti{\rho}^{(n)}_1(t) : = U_* \big(\rho^{(n)}_0(t) + d t \rho^{(n)}_1\big)  
\end{equation}
in $\mMC_{\bul}(\ti{L})$ connecting $U_*(\al^{(n)})$ to $U_*(\al^{(n+1)})$.

Combining \eqref{U-star-rho-n} with the $1$-cell $\ti{\beta}^{(n)}$ in $\mMC_{\bul}(\ti{L})$ 
we get the $1$-dimensional horn in $\mMC_{\bul}(\ti{L})$ shown on figure \ref{1-dim-horn}.
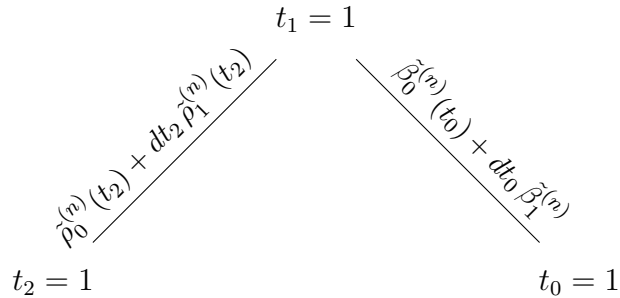
\begin{figure}[htp]
\centering 
\begin{tikzpicture}[scale=0.5]
\tikzstyle{vert}=[circle, minimum size=3, inner sep=4]
%%%%
\node[vert] (v1) at (0, 7) {{$t_1=1$}};
\node[vert] (v0) at (7, 0) {{$t_0 =1$}};
\node[vert] (v2) at (-7, 0) {{$t_2=1$}};
% edges
\draw (v1) edge (v0);
\draw (4.4,3.7) node[anchor=center, rotate = -45] {{\small $\ti{\beta}^{(n)}_0(t_0) + d t_0 \, \ti{\beta}^{(n)}_1 $}};
\draw (v1) edge (v2);
\draw (-4.4,3.7) node[anchor=center, rotate = 45] {{\small $\ti{\rho}^{(n)}_0(t_2) + d t_2  \, \ti{\rho}^{(n)}_1(t_2)$}};
\end{tikzpicture}
\caption{The horn involving the $1$-cells $\ti{\beta}^{(n)}$ and  $\ti{\rho}^{(n)}$} 
\label{1-dim-horn}
\end{figure}

We claim that
\begin{prop}
\label{prop:filling-horn}
There exists a MC element 
$$
\ti{\eta} \in \ti{L} \hotimes \Om_2 
$$
which fills in the horn shown on figure \ref{1-dim-horn} and such that 
the component $ \ti{\eta}'_1(t_0)$ of 
\begin{equation}
\label{del-1-ti-eta}
\ti{\eta}'_0(t_0) + d t_0 \ti{\eta}'_1(t_0) : = 
\ti{\eta} \Big|_{t_1=0} 
\end{equation}
satisfies 
\begin{equation}
\label{eta-pr-1}
\ti{\eta}'_1(t_0) \in \cF_{n+1} \ti{L} \hotimes \bbk[t_0] 
\end{equation}
\end{prop}
\begin{proof}
In this proof we will use the material presented in Appendix \ref{app:Ezra-lemma}. 
For example, we will need the operators $h_1^1$ and $h^1_2$ on    
$\ti{L} \hotimes \Om_1$ and $\ti{L} \hotimes \Om_2$, respectively, and we will need 
Lemma \ref{lem:Ezra}.

According to Lemma \ref{lem:Ezra}, the 1-cell  $\ti{\beta}^{(n)}$ is uniquely determined
by $\ti{\beta}^{(n)}_0(0) = U_*(\al^{(n)})$ and its ``stub''
\begin{equation}
\label{beta-stub}
(\pa + d) h_1^1 (\ti{\beta}^{(n)})  = (\pa + d) (t_0 \ti{\beta}^{(n)}_1)\,.  
\end{equation}
  
Similarly, $1$-cell \eqref{U-star-rho-n} is uniquely determined by $\ti{\rho}^{(n)}_0(0) = U_*(\al^{(n)})$
and its ``stub''
\begin{equation}
\label{ti-rho-stub}
(\pa + d) (\ti{\xi}(t_0)) \,,  
\end{equation}
where 
\begin{equation}
\label{ti-xi-from-ti-rho}
\ti{\xi}(t_0) : = h_1^1 ( \ti{\rho}^{(n)} ) \in \ti{L} \hotimes \bbk[t_0]\,. 
\end{equation}

Since $U$ is compatible with the filtrations and $\rho^{(n)}_1 \in \cF_n L$
$$
\ti{\rho}^{(n)}_1(t_0) - \vf(\rho^{(n)}_1) \in \cF_{n+1} \ti{L}  \hotimes \bbk[t_0]
$$
and hence\footnote{See Remark \ref{rem:rectified}.} 
\begin{equation}
\label{ti-xi-mod-cF-n1}
\ti{\xi}(t_0) - t_0 \vf(\rho^{(n)}_1)  \in \cF_{n+1} \ti{L}  \hotimes \bbk[t_0]\,.
\end{equation}

Our goal is to construct a MC element 
\begin{equation}
\label{ti-eta}
\ti{\eta} \in \ti{L} \hotimes \Om_2 
\end{equation}
such that 
\begin{equation}
\label{ti-eta-ti-beta}
\ti{\eta}  \Big|_{t_2 = 0}  =  \ti{\beta}^{(n)}_0(t_0) + d t_0  \ti{\beta}^{(n)}_1\,,
\end{equation}
\begin{equation}
\label{ti-eta-ti-rho}
\ti{\eta}  \Big|_{t_0 = 0}  =  \ti{\rho}^{(n)}_0(t_2)  +  d t_2 \ti{\rho}^{(n)}_1(t_2)\,,
\end{equation}
and such that condition \eqref{eta-pr-1} is satisfied.

For this purpose, we introduce the following element in\footnote{The notation $\Stub_n^i$ is introduced 
in \eqref{stub} in Appendix \ref{app:Ezra-lemma}.}  $\Stub_2^1(\ti{L})$
\begin{equation}
\label{ti-nu}
\ti{\nu} : = (\pa + d)  \big( \ti{\xi}(t_2) + t_0 \ti{\beta}^{(n)}_1 + (t_2 d t_0 - t_0  d t_2) \ti{\ga} \big)\,,
\end{equation}
where $\ti{\ga}$ is a degree $-2$ element in $\cF_n \ti{L}$
introduced in Claim \ref{cl:rho-1-choice} and $\ti{\xi}$ is defined 
in \eqref{ti-xi-from-ti-rho}. 

It is clear that
$$
\ti{\nu} \Big|_{t_1 = 1} = (\pa + d)  \big( \ti{\xi}(t_2) \big) \Big|_{t_2 = 0} \,.
$$
Furthermore, 
$$
(\pa + d)  \big( \ti{\xi}(t_0) \big) \Big|_{t_0 = 0} = 0
$$
since $(\pa + d)  \big( \ti{\xi}(t_0) \big) $ is the ``stub'' of the $1$-cell $\ti{\rho}^{(n)}$\,.

Thus $\ti{\nu}$ is indeed an element of  $\Stub_2^1(\ti{L})$\,.

Next we set 
\begin{equation}
\label{ti-eta-0}
\ti{\eta}^{(0)} : = U_*(\al^{(n)}) + \ti{\nu} 
\end{equation}
and define $\ti{\eta}$ \eqref{ti-eta} as the limiting element of the 
sequence 
$$
\big\{  \ti{\eta}^{(k)} \big\}_{k \ge 0}
$$ 
defined by the recursive procedure in \eqref{Ezra-recursion}
\begin{equation}
\label{ti-eta-k1}
\ti{\eta}^{(k+1)} : = \ti{\eta}^{(0)} - \sum_{m=2}^{\infty} \frac{1}{m!} \, h^1_2 \, \{ \ti{\eta}^{(k)}, \dots, \ti{\eta}^{(k)}  \}_m\,.
\end{equation}

Conditions \eqref{ti-eta-ti-beta} and \eqref{ti-eta-ti-rho} are satisfied due to Lemma \ref{lem:Ezra}. 

So it remains to prove that 
\begin{equation}
\label{in-cF-n1}
\ti{\eta}'_1(t_0) \in \cF_{n+1} \ti{\cL} \hotimes \bbk[t_0]\,, 
\end{equation}
where $\ti{\eta}'_1(t_0)$ is the degree $-1$ component of 
the $1$-cell
$$
\ti{\eta}'_0(t_0) + d t_0 \ti{\eta}'_1(t_0) : = 
\ti{\eta} \Big|_{t_1=0} \,.
$$

To prove \eqref{in-cF-n1}, we observe that $\ti{\nu} \in  \cF_{n} \ti{\cL}$ and hence 
\begin{equation}
\label{ti-eta-ti-nu}
\ti{\eta} - U_*(\al^{(n)}) - \ti{\nu}  \in  \cF_{n+1} \ti{\cL} \hotimes \Om_2
\end{equation}

On the other hand, 
$$
\ti{\nu} \Big|_{t_1 = 0} = 
(\pa + d)  \big( \ti{\xi}(1-t_0) + t_0 \ti{\beta}^{(n)}_1 +  d t_0  \ti{\ga} \big)\,.
$$
Therefore, using \eqref{ti-xi-mod-cF-n1}, we get 
$$
\ti{\nu} \Big|_{t_1 = 0} -  \Big( \, t_0 \big( \pa \ti{\beta}^{(n)}_1 -  \pa \vf(\rho^{(n)}_1) \big) 
+ d t_0 \big( \ti{\beta}^{(n)}_1 - \vf(\rho^{(n)}_1)  -  \pa \ti{\ga} \big) \, \Big) \in   \cF_{n+1} \ti{\cL} \hotimes \Om_1\,.
$$

Hence 
$$
\ti{\eta}'_1(t_0) -  \big( \ti{\beta}^{(n)}_1 - \vf(\rho^{(n)}_1)  -  \pa \ti{\ga} \big) \in \cF_{n+1} \ti{\cL} \hotimes \bbk[t_0]\,.
$$

Combining the latter inclusion with \eqref{tibeta1-vf-rho1} we immediately deduce 
the desired inclusion $\ti{\eta}'_1(t_0) \in \cF_{n+1} \ti{\cL} \hotimes \bbk[t_0]$\,.

Proposition \ref{prop:filling-horn} is proved. 
\end{proof}

The existence of the desired $1$-cell $\ti{\beta}^{(n+1)}$ in diagram \eqref{diag-induction} 
follows immediately from Proposition \ref{prop:filling-horn} and Lemma \ref{lem:rectify} from Appendix 
\ref{app:rectify}. 

Thus desired sequences  \eqref{al-n-seq}, \eqref{rho-n-seq}, and \eqref{ti-beta-n-seq} are constructed. 

Since the sequence of $1$-cells \eqref{rho-n-seq} satisfies \eqref{rho-n-1} and \eqref{rho-n-begin-end}, 
Lemma \ref{lem:con-ing} from Appendix \ref{app:con-ing} implies that there exists a $1$-cell 
$\ga$ in $\mMC_{\bul}(L)$ which connects $\al=\al^{(1)}$ to the limiting MC element $\al_{lim}$ of 
sequence \eqref{al-n-seq}. 

On the other hand, condition \eqref{al-n-to-al-pr} implies that $\al_{lim} = \al'$.
 
Thus the $0$-cells $\al, \al' \in \mMC_0(L)$ are indeed connected by a $1$-cell.  
The injectivity of map \eqref{pi-0-level} is proved.  

\subsection{Map \eqref{pi-0-level} is surjective}
\label{sec:surjective}

To establish that map \eqref{pi-0-level} is surjective, we will
prove inductively the following proposition:
\begin{prop}
\label{prop:tower-alphas}
For every MC element $\ti{\al}$ in $\ti{L}$, there exists a sequence of 
degree zero elements
\begin{equation}
\label{alalal}
\{ \al^{(n)} \}_{n \ge 0} ~\subset~  L  
\end{equation}
and a sequence of rectified $1$-cells in $\mMC_{\bul}(\ti{L})$ 
\begin{equation}
\label{ti-bebebe}
\ti{\beta}^{(n)} = \ti{\beta}^{(n)}_0(t_0) + d t_0\, \ti{\beta}^{(n)}_1
\end{equation}
such that for every $n \ge 0$
\begin{equation}
\label{curv-al-n}
\curv(\al^{(n)}) \in \cF_{n+2} L\,, 
\end{equation}

\begin{equation}
\label{al-n1-al-n}
\al^{(n+1)} - \al^{(n)} \in \cF_{n+1} L\,, 
\end{equation}

\begin{equation}
\label{begin}
\ti{\beta}^{(n)}_0(0) = \ti{\al}\,,
\end{equation}

\begin{equation}
\label{tibeta-diff-nce}
\ti{\beta}^{(n+1)} - \ti{\beta}^{(n)} \in \cF_{n+1} \ti{L} \hotimes \Om_1
\end{equation}
and 
\begin{equation}
\label{tibeta-n-al-n}
\ti{\beta}^{(n)}_0(1) - U_*(\al^{(n)})  \in \cF_{n+1} \ti{L}\,.
\end{equation}
\end{prop}
\begin{proof}
For $n=0$ we set and 
$$ 
\al^{(0)} = 0 \qquad \text{and} \qquad
\ti{\beta}^{(0)} : = \ti{\al}\,.
$$

In this case, the only non-trivial condition is \eqref{tibeta-n-al-n}
and it holds because $\ti{L} = \cF_1 \ti{L}$\,.

Let us now assume that we constructed elements
$$
\al^{(0)},  \al^{(1)}, \dots,  \al^{(n)}
$$
and 
$$
\ti{\beta}^{(0)}, \ti{\beta}^{(1)}, \dots, \ti{\beta}^{(n)}
$$
with the desired properties. 

Our goal is to construct a degree zero element 
$\al^{(n+1)} \in L$ and a MC element 
$$
\ti{\beta}^{(n+1)} = \ti{\beta}^{(n+1)}_0(t_0) + d t_0\, \ti{\beta}^{(n+1)}_1 ~ \in ~ \ti{L} \hotimes  \Om_1   
$$
such that  \eqref{al-n1-al-n} and \eqref{tibeta-diff-nce} hold, 
\begin{equation}
\label{curv-al-n1}
\curv(\al^{(n+1)}) \in \cF_{n+3} L\,, 
\end{equation} 
\begin{equation}
\label{begin-n1}
\ti{\beta}^{(n+1)}_0(0) = \ti{\al}\,,
\end{equation}
and 
\begin{equation}
\label{tibeta-n1-al-n1}
\ti{\beta}^{(n+1)}_0(1) - U_*(\al^{(n+1)})  \in \cF_{n+2} \ti{L}\,.
\end{equation}

Equation \eqref{U-star-curv} and inclusion \eqref{curv-al-n} imply that 
$$
\curv \big(  U_*(\al^{(n)}) \big) \in \cF_{n+2} \ti{L}\,.
$$
Hence, using  \eqref{curv-sum}, \eqref{tibeta-n-al-n} and the fact 
that $\ti{\beta}^{(n)}$ is a MC element of $\ti{L} \hotimes  \Om_1$
we conclude that
\begin{equation}
\label{tibeta-n-U-al-n}
\ti{\beta}^{(n)}_0(1) - U_*(\al^{(n)})
\end{equation}
represents a cocycle in the quotient complex 
$$
\cF_{n+1} \ti{L} ~\big/~ \cF_{n+2} \ti{L}\,.
$$

Since the map $\vf$ induces quasi-isomorphism \eqref{quotient-n1}, 
there exists a degree $0$ element $\ga \in \cF_{n+1} L$ and a degree $-1$ element 
$\ti{\xi} \in \cF_{n+1} \ti{L}$ such that 
\begin{equation}
\label{diff-ga}
\pa(\ga) \in  \cF_{n+2} L
\end{equation}
and 
\begin{equation}
\label{in-cF-n2}
\ti{\beta}^{(n)}_0(1) + \pa(\ti{\xi})  - \big( U_*(\al^{(n)}) +  \vf(\ga) \big) \in \cF_{n+2} \ti{L}\,.
\end{equation}

So we introduce the new degree $0$ element 
\begin{equation}
\label{alpha}
\al : = \al^{(n)} + \ga \in L
\end{equation}
and the new (rectified) $1$-cell 
\begin{equation}
\label{tibeta-n1}
\ti{\beta}^{(n+1)} = \ti{\beta}^{(n+1)}_0(t_0) +  d t_0 \ti{\beta}^{(n+1)}_1\,,
\end{equation}
in $\mMC_{\bul}(\ti{L})$, where 
\begin{equation}
\label{tibeta-n1-1}
\ti{\beta}^{(n+1)}_1 = \ti{\beta}^{(n)}_1 + \ti{\xi}\,,
\end{equation}
$\ti{\beta}^{(n+1)}_0(t_0)$ is the unique solution of the integral equation 
\begin{equation}
\label{tibeta-n1-0}
\ti{\beta}^{(n+1)}_0(t_0) =  \ti{\al} + \int_0^{t_0} \pa^{\,\ti{\beta}^{(n+1)}_0(u)\,} (\ti{\beta}^{(n)}_1 + \ti{\xi}) d u\,,
\end{equation}
and  $\pa^{\,\ti{\beta}^{(n+1)}_0(u)\,}$ denotes the differential twisted by the MC element 
$\ti{\beta}^{(n+1)}_0(u)$.

Inclusions \eqref{curv-al-n},  $\ga \in \cF_{n+1} L$, and \eqref{diff-ga} imply that 
\begin{equation}
\label{curv-al-cF-n2}
\curv(\al) \in \cF_{n+2} L\,.  
\end{equation}
Similarly, inclusions $\ti{\xi} \in \cF_{n+1} \ti{L}$ and $\ga \in \cF_{n+1} L$  imply that 
\begin{equation}
\label{new-end}
\ti{\beta}^{(n)}_0(1) + \pa(\ti{\xi}) - \ti{\beta}^{(n+1)}_0(1) \in  \cF_{n+2} \ti{L}
\end{equation}
and  
\begin{equation}
\label{U-star-al}
U_*(\al^{(n)}) +  \vf(\ga) - U_*(\al)  \in  \cF_{n+2} \ti{L}\,.
\end{equation}

Thus, combining  \eqref{in-cF-n2}, \eqref{new-end}, and \eqref{U-star-al}, 
we conclude that 
\begin{equation}
\label{new-end-U-alpha}
\ti{\beta}^{(n+1)}_0(1)  -  U_*(\al)  \in \cF_{n+2} \ti{L}\,.
\end{equation} 
 
We will now use $\al$ to construct a degree $0$ element
$\al^{(n+1)}$ satisfying \eqref{al-n1-al-n}, \eqref{curv-al-n1}, and 
\eqref{tibeta-n1-al-n1}. 
 
Due to  \eqref{Bianchi} and \eqref{curv-al-cF-n2}, the degree $1$ element 
$\curv(\al)$ represents a cocycle in the quotient complex 
$$
\cF_{n+2} L \big/ \cF_{n+3} L\,.
$$

Let us prove that 
\begin{claim}
\label{cl:curv-al-exact}
The element $\curv(\al)$ represents a coboundary in the quotient complex 
$\cF_{n+2} L \big/ \cF_{n+3} L\,.$
\end{claim}
\begin{subproof}[ of Claim \ref{cl:curv-al-exact}]
Due to \eqref{U-star-curv} 
$$
\curv\big( U_*(\al) \big) = 
\sum_{m \ge 0} \frac{1}{m!} U' \big( \al^m  \curv(\al) \big) \in \cF_{n+2} \ti{L}\,. 
$$
Hence, inclusion \eqref{curv-al-cF-n2} implies that 
\begin{equation}
\label{vf-curv-al}
\vf(\curv(\al)) - \curv\big( U_*(\al) \big)  \in \cF_{n+3} \ti{L}\,.
\end{equation}

On the other hand, equation \eqref{curv-sum} implies that 
$$
\curv\big( U_*(\al) \big) =  \curv\big(\, \ti{\beta}^{(n+1)}_0(1)  +  (U_*(\al) - \ti{\beta}^{(n+1)}_0(1)  )\, \big) =
$$
$$
\curv( \ti{\beta}^{(n+1)}_0(1) ) + \curv^{\ti{\beta}^{(n+1)}_0(1)}  \big( U_*(\al) - \ti{\beta}^{(n+1)}_0(1) \big)\,,
$$
where 
$$
\curv^{\beta}{\al} : = \pa^{\beta}(\al) + \sum_{m=2}^{\infty} \frac{1}{m!} \{\al, \dots, \al \}^{\beta}_m\,,   
$$
and $\pa^{\beta}$ and $\{ \cdot, \dots, \cdot \}^{\beta}_m$ denote the differential and the multi-brackets of the 
$\sLie$-structure twisted by $\beta$. 

Hence, using  \eqref{new-end-U-alpha} and 
the fact that $\ti{\beta}^{(n+1)}_0(1)$ is a MC element of $\ti{L}$ we 
conclude that 
$$
\curv\big( U_*(\al) \big) -  \pa  \big( U_*(\al) - \ti{\beta}^{(n+1)}_0(1) \big) ~ \in ~ \cF_{n+3} \ti{L}\,.
$$

Combining this observation with inclusion \eqref{vf-curv-al}, we conclude that 
the element $\vf(\curv(\al))$ represents a coboundary in the quotient complex 
$$
\cF_{n+2} \ti{L} \big/ \cF_{n+3} \ti{L}\,.
$$ 

Thus the desired statement follows from the fact that 
$\vf$ induces quasi-isomorphism \eqref{quotient-n1}. 
\end{subproof}

Due to Claim \ref{cl:curv-al-exact}, there exists a degree $0$ element 
$\si \in \cF_{n+2} L$ such that
$$
\curv(\al) + \pa (\si) \in \cF_{n+3} L\,.
$$ 

Thus, setting 
$$
\al^{(n+1)} : = \al + \si
$$
we get a degree $0$ element of $L$ satisfying desired 
properties \eqref{al-n1-al-n}, \eqref{curv-al-n1}, and 
\eqref{tibeta-n1-al-n1}. 

Proposition \ref{prop:tower-alphas} is proved.
\end{proof}

Let us denote by $\al_{lim} \in L $ and $\ti{\beta} \in \ti{L} \hotimes \Om_1$ 
the limiting elements of sequences \eqref{alalal} and \eqref{ti-bebebe}, respectively.

Due equation \eqref{begin} and 
inclusions \eqref{curv-al-n}, \eqref{tibeta-n-al-n}, $\al_{lim}$ is a MC 
element of $L$ and $\ti{\beta}$ is a $1$-cell of $\mMC_{\bul}(\ti{L})$ which 
connects the MC elements $\ti{\al}$ and $U_*(\al_{lim})$ of $\ti{L}$. 

Thus we proved that map \eqref{pi-0-level} is surjective. 

\section{Taking care of higher homotopy groups}
\label{sec:higher}

Let us start this section by recalling a lemma from \cite{EnhancedLie}
which allows us to reduce questions about homotopy groups
$\pi_i (\mMC_{\bul}(L), \al) $ of $\mMC_{\bul}(L)$ 
with an arbitrary base point $\al \in \MC(L)$  to 
the corresponding questions about  homotopy groups of
$\mMC_{\bul}(L^{\al})$ with the zero base point:
\begin{lem}[Lem.\ 4.3 \cite{EnhancedLie}]
\label{lem:zero-base-pt}
Let $\al$ be a MC element in $L$ and $L^{\al}$ be the filtered $\sLie$ algebra 
which is obtained from $L$ via twisting by $\al$. Then the following assignment 
\begin{equation}
\label{eq:zero-base}
\beta \in \MC\big( L^{\al} \hotimes \Om_n \big)  ~~\mapsto ~~
\al + \beta \in \MC\big( L \hotimes \Om_n \big)
\end{equation}
is an isomorphism of simplicial sets 
\begin{equation}
\label{Shift-al}
\Shift_{\al} : \mMC_{\bul}(L^{\al})  \to  \mMC_{\bul}(L)
\end{equation}
which sends the zero MC element of $L^{\al}$ to the MC 
element $\al$ in $L$. For every $\infty$-morphism $U$ of filtered 
$\sLie$-algebras $L \to \ti{L}$ the following diagram commutes: 
\begin{equation}
\label{diag-Shift}
\begin{tikzpicture}
\matrix (m) [matrix of math nodes, row sep=2.5em, column sep=4em]
{    \mMC_{\bul}(L^{\al})  &  \mMC_{\bul}(L)  \\
  \mMC_{\bul}(\,\ti{L}^{U_*(\al)}\,)  &  \mMC_{\bul}(\,\ti{L}\,)\,, \\ };
\path[->,font=\scriptsize]
(m-1-1) edge node[auto] {$\Shift_{\al}$}  (m-1-2)  edge node[left] {$\mMC_{\bul}(U^{\al})$} (m-2-1)
(m-1-2)  edge node[auto] {$\mMC_{\bul}(U)$}  (m-2-2) 
(m-2-1) edge node[auto] {$\Shift_{U_*(\al)}$} (m-2-2);
\end{tikzpicture}
\end{equation}
where $U^{\al}$ denotes the $\infty$-morphism $L^{\al} \to \ti{L}^{U_*(\al)}$  which is obtained 
from $U$ via twisting by the MC element $\al$. 
\end{lem}
% \begin{proof} 
% Eq.\ \eqref{curv-sum} from Prop.\ \ref{prop:curv} implies that map \eqref{eq:zero-base} is well defined, and it is clearly
% injective. If $\gamma \in \MC\big( L \hotimes \Om_n \big)$, then
% $\curv(\gamma) = \curv\bigl(\alpha + (\gamma -\alpha)
% \bigr)=0$. Hence, \eqref{curv-sum} implies that
% \[
% \gamma -\alpha \in \MC\big( L^{\al} \hotimes \Om_n \big), 
% \] 
% so map \eqref{eq:zero-base} is also surjective. Since $\alpha$ is
% constant as an element of $L \hotimes \Omega_\bullet$,
% \eqref{eq:zero-base} induces the isomorphism of simplicial sets
% \eqref{Shift-al}.

% To show that diagram \eqref{diag-Shift} commutes,
% suppose that $\beta \in \mMC_{\bul}(L^{\al})$. Then we have
% \begin{equation} \label{eq:shift-sum1}
% \begin{split}
% \bigl(\mMC_{\bul}(U) \circ \Shift_{\al} \bigr) (\beta) &= U_{\ast}(\al + \beta) 
% = \sum_{k \geq 1} \frac{1}{k!} U'\bigl( (\alpha + \beta)^k \bigr) 
% =\sum_{k \geq 1} \sum_{l=0}^{k} \frac{1}{l! (k-l)!} U' \bigl(\alpha^l
% \beta^{k-l} \bigr).
% \end{split}
% \end{equation}
% On the other hand,
% \begin{equation}
% \begin{split}
% \bigl(\Shift_{U_{\ast}(\alpha)} \circ \mMC_{\bul}(U^{\al}) \bigr)(\beta) &=
% U_{\ast}(\alpha) + U^{\al}_{\ast}(\beta) = \sum_{l \geq 1} \frac{1}{l!} U'(\alpha^l) + \sum_{k \geq 1}
% \frac{1}{k!}(U^{\alpha})' (\beta^k)\\
% & =\sum_{l \geq 1} \frac{1}{l!} U'(\alpha^l) + \sum_{k \geq 1} \sum_{l
%   \geq 0} \frac{1}{l!k!}U'(\alpha^l \beta^k).
% \end{split}
% \end{equation}
% After rearranging terms, we see that the two compositions above are equal.
% \end{proof}

Let us now prove that for every $n \ge 2$ the map $\mMC_{\bul}(U)$ induces 
a weak equivalence of simplicial sets:
\begin{equation}
\label{for-L-factor-cF-n-L}
\mMC_{\bul} \big(L / \cF_n L \big) \longrightarrow  \mMC_{\bul} \big( \ti{L} / \cF_n \ti{L} \big)
\end{equation}
  
Since $\sLie$-algebras $L \big/ \cF_2 L$ and  $\ti{L} \big/ \cF_2 \ti{L}$
are abelian, Proposition \ref{prop:Abelian} implies that \eqref{for-L-factor-cF-n-L} is 
indeed a weak equivalence for $n =2$. 

So we take $n = 2$ as the base of our induction and assume that 
$$
\mMC_{\bul} \big(L / \cF_m L \big) \longrightarrow  \mMC_{\bul} \big( \ti{L} / \cF_m \ti{L} \big)
$$
is a weak equivalence of simplicial sets for every $m \le n$. 

Our goal is to prove that 
\begin{equation}
\label{goal}
\mMC_{\bul} \big(L / \cF_{n+1} L \big) \longrightarrow  \mMC_{\bul} \big( \ti{L} / \cF_{n+1} \ti{L} \big)
\end{equation}
is also a weak equivalence of simplicial sets. 

Due to the results of Section \ref{sec:pi0-level} and Lemma \ref{lem:zero-base-pt}, 
it suffices to prove that 
\begin{equation}
\label{i-ge-1}
\pi_i \Big( \mMC_{\bul} \big(L / \cF_{n+1} L \big), ~ 0 \Big) 
\longrightarrow  
\pi_i \Big( \mMC_{\bul} \big( \ti{L} / \cF_{n+1} \ti{L} \big), ~ 0 \Big)
\end{equation}
is an isomorphism of groups for every $i \ge 1$\,.

For this purpose, we consider the following commutative diagram 
\begin{equation}
\label{n-n1-diagram}
\begin{tikzpicture}
\matrix (m) [matrix of math nodes, row sep=2em, column sep=2em]
{\bfzero  &  \cF_n L / \cF_{n+1} L  &  L / \cF_{n+1} L  &  L / \cF_n L  & \bfzero \\
\bfzero  & \cF_n \ti{L} / \cF_{n+1} \ti{L}  &  \ti{L} / \cF_{n+1}  \ti{L}  &  \ti{L} / \cF_n  \ti{L} & \bfzero\,, \\  };
\path[->,font=\scriptsize]
(m-1-1) edge (m-1-2)  (m-1-2) edge (m-1-3)
edge node[left] {} (m-2-2) (m-1-3) edge (m-1-4)
edge node[left] {} (m-2-3) (m-1-4) edge (m-1-5) 
edge node[left] {} (m-2-4) 
(m-2-1) edge (m-2-2)  (m-2-2) edge  (m-2-3)  (m-2-3) edge 
 (m-2-4)  (m-2-4) edge  (m-2-5);    
\end{tikzpicture}
\end{equation}
where the rows form short exact sequences of strict morphisms of $\sLie$-algebras, 
the left most vertical arrow is the strict morphism induced by the 
linear term $\vf$ of $U$, and the other two vertical arrows are 
$\infty$-morphisms induced by $U$. 

Since the morphisms  $ L / \cF_{n+1} L \to   L / \cF_{n} L$ and  
$ \ti{L} / \cF_{n+1} \ti{L} \to   \ti{L} / \cF_{n} \ti{L}$ are surjective,
Proposition 4.7 from \cite{Ezra-infty} implies that the 
corresponding maps 
\begin{equation}
\label{mod-cF-n}
\mMC_{\bul}(L / \cF_{n+1} L) \to   \mMC_{\bul}(L / \cF_{n} L) 
\quad \textrm{and} \quad 
\mMC_{\bul}(\ti{L} / \cF_{n+1} \ti{L}) \to   \mMC_{\bul}(\ti{L} / \cF_{n} \ti{L})
\end{equation}
are fibrations of simplicial sets with the fibers over the base point 
$0 \in  L / \cF_{n} L$ (resp. $0 \in  \ti{L} / \cF_{n} \ti{L}$) being 
$\mMC_{\bul}(\cF_n L / \cF_{n+1} L )$ (resp. $\mMC_{\bul}( \cF_n \ti{L} / \cF_{n+1} \ti{L})$). 

Therefore, the rows of diagram \eqref{n-n1-diagram} give us the 
long ``exact'' sequences\footnote{The words exact is put in quotation marks because
the $\pi_0$ portion of sequence \eqref{les-homot-gr} is not an exact sequence 
of groups. Moreover, in general, the map  $\pi_0 ( L / \cF_{n+1} L) \to   \pi_0 ( L / \cF_{n} L)$
is not a surjective map of sets.}
of homotopy groups (see, for example, \cite[Lemma 7.3]{Goerss-Jardine}): 
\begin{equation}
\label{les-homot-gr}
\begin{tikzpicture}
\matrix (m) [matrix of math nodes, row sep=2em, column sep=1em]
{ \dots & \pi_1(\cF_n L / \cF_{n+1} L) & \pi_1 ( L / \cF_{n+1} L) & \pi_1(L / \cF_{n} L) \\ 
~ & \pi_0 (\cF_n L / \cF_{n+1} L) &  \pi_0 ( L / \cF_{n+1} L) &  \pi_0 ( L / \cF_{n} L)\,, \\ };
\path[->,font=\scriptsize]
(m-1-1) edge (m-1-2)  (m-1-2) edge (m-1-3) (m-1-3) edge (m-1-4) 
(m-1-4)  edge[out=355, in=175] (m-2-2) (m-2-2) edge  (m-2-3)  (m-2-3) edge  (m-2-4);
\end{tikzpicture}
\end{equation} 
%%%
\begin{equation}
\label{les-homot-gr-tilde}
\begin{tikzpicture}
\matrix (m) [matrix of math nodes, row sep=2em, column sep=1em]
{ \dots & \pi_1(\cF_n \ti{L} / \cF_{n+1} \ti{L}) & \pi_1 ( \ti{L} / \cF_{n+1} \ti{L}) & \pi_1(\ti{L} / \cF_{n} \ti{L}) \\ 
~ & \pi_0 (\cF_n \ti{L} / \cF_{n+1} \ti{L}) &  \pi_0 ( \ti{L} / \cF_{n+1} \ti{L}) &  \pi_0 ( \ti{L} / \cF_{n} \ti{L})\,, \\ };
\path[->,font=\scriptsize]
(m-1-1) edge (m-1-2)  (m-1-2) edge (m-1-3) (m-1-3) edge (m-1-4) 
(m-1-4)  edge[out=355, in=175] (m-2-2) (m-2-2) edge  (m-2-3)  (m-2-3) edge  (m-2-4);
\end{tikzpicture}
\end{equation} 
where by abuse of notation 
$$
\pi_0(L): = \pi_0 (\mMC_{\bul}(L))\,, \qquad 
\pi_i(L) : = \pi_i (\mMC_{\bul}(L), 0)\,, \qquad i \ge 1\,.
$$

Furthermore, the $\infty$-morphism $U$ induces a map from sequence 
\eqref{les-homot-gr} to sequence \eqref{les-homot-gr-tilde}. 

For our purposes, we only need certain truncations of 
sequences \eqref{les-homot-gr}, \eqref{les-homot-gr-tilde}.
So we denote by 
\begin{equation}
\label{mb-L-n-n1}
\mb(L, \cF_n L , \cF_{n+1} L)
\end{equation}
the image of the map 
$$
\pi_1(L / \cF_{n} L) \to  \pi_0 (\cF_n L / \cF_{n+1} L)\,.
$$
In other words, $\mb(L, \cF_n L , \cF_{n+1} L)$ consists of 
cohomology classes $c \in H^0(\cF_n L / \cF_{n+1} L)$ 
satisfying the following property: 
\begin{pty}
\label{P:in-mb}
For $c$, there exists a MC element 
$$
\beta  = \beta_0(t_0)   + d t_0 \beta_1(t_0) ~\in~ \big( L / \cF_{n+1} L \big) \hotimes  \Om_1
$$
such that $\beta_0(0) =0$ in $L / \cF_{n+1} L$ and $\beta_0(1)$ is a cocycle in 
$ \cF_n L / \cF_{n+1} L$ representing the class $c$. 
\end{pty}

Let us now consider the following commutative diagram  
\begin{equation}
\label{homotopy-master}
\begin{tikzpicture}
\matrix (m) [matrix of math nodes, row sep=2em, column sep=1em]
{ \dots & \pi_1(\cF_n L / \cF_{n+1} L) & \pi_1 ( L / \cF_{n+1} L) & \pi_1(L / \cF_{n} L) &  \mb(L, \cF_n L , \cF_{n+1} L)\,, \\ 
 \dots & \pi_1(\cF_n \ti{L} / \cF_{n+1}  \ti{L}) & \pi_1 (  \ti{L} / \cF_{n+1}  \ti{L}) & \pi_1( \ti{L} / \cF_{n}  \ti{L}) 
 &  \mb( \ti{L}, \cF_n  \ti{L} , \cF_{n+1}  \ti{L})\,,  \\ };
\path[->,font=\scriptsize]
(m-1-1) edge (m-1-2)  (m-1-2) edge (m-1-3) (m-1-3) edge (m-1-4) 
(m-1-4)  edge (m-1-5)
 (m-1-2) edge  (m-2-2) (m-1-3) edge  (m-2-3)  (m-1-4) edge  (m-2-4)
  (m-1-5) edge  (m-2-5)
 (m-2-1) edge (m-2-2)  (m-2-2) edge (m-2-3) (m-2-3) edge (m-2-4) 
(m-2-4)  edge (m-2-5);
\end{tikzpicture}
\end{equation} 
where vertical arrows are induced by $\mMC_{\bul}(U)$, rows 
\begin{equation}
\label{no-tilde-row}
\begin{tikzpicture}
\matrix (m) [matrix of math nodes, row sep=2em, column sep=1em]
{ \dots &  \pi_2(L / \cF_{n} L) & \pi_1(\cF_n L / \cF_{n+1} L) & \pi_1 ( L / \cF_{n+1} L) & \pi_1(L / \cF_{n} L)   \\ };
\path[->,font=\scriptsize]
(m-1-1) edge (m-1-2)  (m-1-2) edge (m-1-3) (m-1-3) edge (m-1-4) 
(m-1-4)  edge (m-1-5);
\end{tikzpicture}
\end{equation}
%%%
\begin{equation}
\label{tilde-row}
\begin{tikzpicture}
\matrix (m) [matrix of math nodes, row sep=2em, column sep=1em]
{ \dots &  \pi_2(\ti{L} / \cF_{n} \ti{L}) & \pi_1(\cF_n \ti{L} / \cF_{n+1} \ti{L}) & \pi_1 ( \ti{L} / \cF_{n+1} \ti{L}) & \pi_1(\ti{L} / \cF_{n} \ti{L})   \\ };
\path[->,font=\scriptsize]
(m-1-1) edge (m-1-2)  (m-1-2) edge (m-1-3) (m-1-3) edge (m-1-4) 
(m-1-4)  edge (m-1-5);
\end{tikzpicture}
\end{equation}
are sequences of groups which are exact in terms 
$$
 \pi_1 ( L / \cF_{n+1} L),  ~\pi_1(\cF_n L / \cF_{n+1} L),~  \pi_2(L / \cF_{n} L), ~\dots 
$$
and
$$
 \pi_1 ( \ti{L} / \cF_{n+1}  \ti{L}),  ~\pi_1(\cF_n  \ti{L} / \cF_{n+1}  \ti{L}),~  \pi_2( \ti{L} / \cF_{n}  \ti{L}), ~\dots \,,
$$
respectively.

We claim that 
\begin{prop}
\label{prop:induction}
If $\mMC_{\bul}(U)$ induces an isomorphism of groups  
$$
\pi_i \big(\mMC_{\bul}(L/\cF_n L), 0 \big) ~\to~ \pi_i\big( \mMC_{\bul}(\,\ti{L}/ \cF_n \ti{L} \, ), 0 \big) 
$$ 
for every $i \ge 1$ then the corresponding homomorphism of groups 
$$
\pi_i \big(\mMC_{\bul}(L/\cF_{n+1} L), 0 \big) ~\to~ \pi_i\big( \mMC_{\bul}(\,\ti{L}/ \cF_{n+1} \ti{L} \, ), 0 \big) 
$$ 
is also an isomorphism for all $i \ge 1$. 
\end{prop}
\begin{proof}
The most subtle part of the proof is the surjectivity of the homomorphism
\begin{equation}
\label{mod-n1-mod-n}
\pi_1(L/\cF_{n+1} L) \to  \pi_1(\ti{L}/\cF_{n+1} \ti{L})\,.
\end{equation}
So we will start with proving this fact. 

Let $\ti{g}$ be an element of $ \pi_1(\ti{L}/\cF_{n+1} \ti{L})$ and 
$\ti{h}$ be the corresponding element in $\pi_1(\ti{L}/\cF_{n} \ti{L})$\,.

Since the homomorphism
\begin{equation}
\label{pi1-mod-n}
\pi_1(L/\cF_{n} L) \to  \pi_1(\ti{L}/\cF_{n} \ti{L})
\end{equation}
is an isomorphism there exists (a unique) element 
$h \in \pi_1(L/\cF_{n} L)$ which is sent to $\ti{h}$ via
\eqref{pi1-mod-n}.  

Let us denote by $c$ (resp. $\ti{c}$) the image of $h$ (resp. $\ti{h}$) in $\mb(L, \cF_n L , \cF_{n+1} L)$
(resp. $\mb(\ti{L}, \cF_n \ti{L} , \cF_{n+1} \ti{L})$). 

We know that $c$ is a cohomology class in $H^0(\cF_{n} L / \cF_{n+1} L)$
for which there exists a MC element 
\begin{equation}
\label{beta-mod-n1}
\beta  = \beta_0(t)   + d t \beta_1(t) ~\in~ \big( L / \cF_{n+1} L \big) \hotimes  \Om_1
\end{equation}
such that 
\begin{itemize}

\item $\beta_0(0) =0$ in $L / \cF_{n+1} L$\,,

\item $\beta_0(1)$ is a cocycle in $ \cF_n L / \cF_{n+1} L$ representing the class $c$, 

\item the image of $1$-cell \eqref{beta-mod-n1} in $\mMC_{\bul}( L / \cF_{n} L  )$ is 
a loop representing $h \in  \pi_1(L/\cF_{n} L)$\,.

\end{itemize}

Since $\ti{h}$ comes from an element of $\ti{g}\in  \pi_1(\ti{L}/\cF_{n+1} \ti{L})$, 
the class $\ti{c} \in H^0(\cF_{n} \ti{L} / \cF_{n+1} \ti{L})$ is zero. Hence, since 
diagram \eqref{homotopy-master} commutes and $\vf$ induces an isomorphism 
$$
H^{\bul}(\cF_{n} L / \cF_{n+1} L)  \to  H^{\bul}(\cF_{n} \ti{L} / \cF_{n+1} \ti{L})
$$ 
the class $c$ is also zero.  

Therefore, there exists an element $\xi \in  \cF_n L / \cF_{n+1} L$ such that 
\begin{equation}
\label{beta-1-diff-xi}
\beta_0(1) + \pa \xi  = 0  ~~~\textrm{in}~~~   \cF_n L / \cF_{n+1} L\,. 
\end{equation}

Using the inclusion $\xi \in  \cF_n L / \cF_{n+1} L$ and \eqref{beta-1-diff-xi},  
 it is easy to 
show that the degree $0$ element 
\begin{equation}
\label{beta-pr}
\beta' : = \beta_0(t) + t \pa \xi + d t (\beta_1(t) + \xi) ~\in~ \big( L / \cF_{n+1} L \big) \hotimes  \Om_1
\end{equation}
satisfies these properties: 
\begin{itemize}

\item[i)] $\beta'$ is a MC element in $ \big( L / \cF_{n+1} L \big) \hotimes  \Om_1 $\,,

\item[ii)]  $\beta' \big|_{t=0} = \beta' \big|_{t=1} = 0$\,, and 

\item[iii)] the image of $1$-cell $\beta'$ in $\mMC_{\bul}( L / \cF_{n} L  )$ is 
a loop representing $h \in  \pi_1(L/\cF_{n} L)$\,.

\end{itemize}

On the other hand, property ii) means that the $1$-cell $\beta'$ is also 
a loop in $\mMC_{\bul}( L / \cF_{n+1} L  )$. Thus we proved that there exists 
an element 
\begin{equation}
\label{g-prime}
g' \in \pi_1( L / \cF_{n+1} L  )
\end{equation}
whose image in $\pi_1( L / \cF_{n} L  )$ is $h$.

Let us denote by $\ti{g}'$ the image of $g'$ in  $\pi_1(\, \ti{L} / \cF_{n+1} \ti{L} \, )$\,.
Since diagram \eqref{homotopy-master} commutes, we the element 
\begin{equation}
\label{ti-g-pr-tig}
\ti{g}' \ti{g}^{-1} 
\end{equation}
belongs to the kernel of the map   
$$
\pi_1(\, \ti{L} / \cF_{n+1} \ti{L} \, ) \to \pi_1(\, \ti{L} / \cF_{n} \ti{L} \, )\,.
$$

Therefore, since sequence \eqref{tilde-row} is exact in $\pi_1(\, \ti{L} / \cF_{n+1} \ti{L} \, )$ 
and 
$$
\pi_1(\cF_n L  / \cF_{n+1} L ) \to 
\pi_1(\, \cF_n \ti{L} / \cF_{n+1} \ti{L} \, )
$$ 
is an isomorphism of groups, there exists $f \in \pi_1(\cF_n L  / \cF_{n+1} L )$
which is sent to $\ti{g} (\ti{g}')^{-1}$ via the composition 
$$
\pi_1(\cF_n L  / \cF_{n+1} L ) \to  \pi_1( L  / \cF_{n+1} L ) \to \pi_1(\, \ti{L} / \cF_{n+1} \ti{L} \, )\,.
$$

Thus, if $f'$ is the image of $f$ in $\pi_1( L  / \cF_{n+1} L )$ then the image of $f'g'$ in 
$\pi_1(\, \ti{L} / \cF_{n+1} \ti{L} \, )$ coincides with $\ti{g}$. We proved that homomorphism
\eqref{mod-n1-mod-n} is surjective. 

The proof of injectivity of  \eqref{mod-n1-mod-n} is much easier so we leave it to the reader. 

It remains to prove that 
$$
\pi_i \big(\mMC_{\bul}(L/\cF_{n+1} L), 0 \big) ~\to~ \pi_i\big( \mMC_{\bul}(\,\ti{L}/ \cF_{n+1} \ti{L} \, ), 0 \big) 
$$ 
is an isomorphism for all $i \ge 2$. This is done by induction on $i$ using the exactness of 
sequences \eqref{no-tilde-row} and \eqref{tilde-row} in terms
$$
 \pi_1 ( L / \cF_{n+1} L),  ~\pi_1(\cF_n L / \cF_{n+1} L),~  \pi_2(L / \cF_{n} L), ~\dots 
$$
and
$$
 \pi_1 ( \ti{L} / \cF_{n+1}  \ti{L}),  ~\pi_1(\cF_n  \ti{L} / \cF_{n+1}  \ti{L}),~  \pi_2( \ti{L} / \cF_{n}  \ti{L}), ~\dots \,,
$$ 
respectively. Here we also use the fact that $\mMC_{\bul}(U)$ induces an isomorphism of groups 
$$
\pi_i \big(\mMC_{\bul}(\cF_n L / \cF_{n+1} L), 0 \big) ~\to~ \pi_i\big( \mMC_{\bul}(\,\cF_n \ti{L}/ \cF_{n+1} \ti{L} \, ), 0 \big) 
$$ 
for every $i \ge 1$. This fact follows from Proposition \ref{prop:Abelian} since both $\cF_n L / \cF_{n+1} L$ 
and $\cF_n \ti{L}/ \cF_{n+1} \ti{L}$ are abelian $\sLie$-algebras. 
  
Proposition \ref{prop:induction} is proved. 
\end{proof}

\subsection{The end of the proof of Theorem \ref{thm:GM}}
\label{sec:the-end}
The results presented in Section \ref{sec:pi0-level}, Lemma \ref{lem:zero-base-pt}, and Proposition \ref{prop:induction} imply that $U$ induces a morphism between towers
of Kan complexes:
\[
\begin{tikzpicture}[descr/.style={fill=white,inner sep=2.5pt},baseline=(current  bounding  box.center)]
\matrix (m) [matrix of math nodes, row sep=2em,column sep=3em,
  ampersand replacement=\&]
  {  
\vdots \& \vdots\\
\mMC_{\bul}(L / \cF_{n+1} L) \& \mMC_{\bul}(\ti{L} / \cF_{n+1} \ti{L})\\
\mMC_{\bul}(L / \cF_{n} L) \& \mMC_{\bul}(\ti{L} / \cF_{n} \ti{L})\\
\mMC_{\bul}(L / \cF_{n-1} L) \& \mMC_{\bul}(\ti{L} / \cF_{n-1} \ti{L})\\
\vdots \& \vdots\\
}; 
     \path[->,font=\scriptsize] 
  (m-1-1) edge node[auto]{$ $}  (m-2-1)
  (m-1-2) edge node[auto]{$ $}  (m-2-2)
  (m-2-1) edge  node[auto]{$ $}  (m-3-1)
  (m-2-2) edge node[auto]{$ $}  (m-3-2)
  (m-3-1) edge  node[auto]{$ $}  (m-4-1)
  (m-3-2) edge node[auto]{$ $}  (m-4-2)
  (m-4-1) edge  node[auto]{$ $}  (m-5-1)
  (m-4-2) edge node[auto]{$ $}  (m-5-2)
 (m-2-1) edge  node[auto]{$\sim$}  (m-2-2)
  (m-3-1) edge  node[auto]{$\sim$}  (m-3-2)
  (m-4-1) edge  node[auto]{$\sim$}  (m-4-2)
;
\end{tikzpicture}
\]
in which every vertical arrow is a fibration and every horizontal
arrow is a weak equivalence. Hence, we have  a weak equivalence
between fibrant objects in the model category of towers of simplicial
sets (Definition 1.1, Proposition 1.3, and Remark 1.5  in \cite[Section VI]{Goerss-Jardine} ). The
inverse limit is right adjoint to the constant tower functor, which
preserves cofibrations and weak equivalences. Hence, the inverse limit preserves
weak equivalences between fibrant objects, and thus
we conclude that the map
\[
\mMC_{\bul}(U) \colon \varprojlim \mMC_{\bul}(L / \cF_{n} L) \to
\varprojlim \mMC_{\bul}(\ti{L} / \cF_{n} \ti{L})
\]
is a homotopy equivalence of simplicial sets.

\appendix

\section{The homotopy operators $h^i_n$ and Lemma 4.6 from \cite{Ezra-infty}}
\label{app:Ezra-lemma}

Let, as above, $\Delta^{n}$ denote the geometric $n$-simplex 
and $\Omega_{n}$ denote Sullivan's polynomial de Rham
complex of $\Delta^{n}$ with coefficients in $\bbk$. 

For $0 \leq i \leq n$, let $\phi_{i} \maps [0,1] \times
\Delta^{n} \to \Delta^n$ be the map
\[
\phi_{i}(\, u,\vec{t}\,) = u\vec{t} + (1-u) \vec{e}_i,
\]
where $\vec{e}_i$ denotes the $i$th vertex of $\Delta^{n}$. Let 
\[
\pi_{\ast} \maps \Omega^{\bullet}\bigl([0,1] \times \Delta^{n} \bigr)
\to \Omega^{\bullet -1}_n
\]
denote the integration over the fibers of the projection $\pi \maps 
[0,1] \times \Delta^{n} \to \Delta^n$. As in \cite[Sec. 4]{Ezra-infty}, we define the homotopy operators
\begin{equation}
\label{eq:ho_op1}
\begin{split}
h^{i}_{n} \maps \Omega^{\bullet}_{n} \to \Omega^{\bullet-1}_{n}\\
h^{i}_{n} \omega = \pi_{\ast} \phi^{\ast}_i \omega.
\end{split}
\end{equation}
Explicitly, if 
$$
\omega = f(t_0, \ldots, t_n) dt_{k_1} dt_{k_2} \cdots dt_{k_m}
$$ 
is a $m$-form on $\Delta^{n}$, then
\begin{equation} 
\label{eq:ho_op2}
h^{i}_{n} \omega = \sum_{j=1}^{m} (-1)^{j-1} (t_{k_{j}} -\delta_{ik_{j}}) 
dt_{k_1} dt_{k_2} \cdots \widehat{dt_{k_{j}}} \cdots dt_{k_m}
\int^{1}_{0} u^{m-1} (f\circ \phi_{i}) ~ du,
\end{equation}
where $\delta_{ik_{j}}$ is the Kronecker delta. 

One can show that $h^i_{n}$ is a chain homotopy between
$\ve^i_n$ and the identity $\id_{\Omega_n}$:
\begin{equation} 
\label{eq:poincare}
d h^{i}_{n} +h^{i}_{n} d = \id_{\Omega_n} - \ve^i_n,
\end{equation}
where $\ve^{i}_{n} \maps \Omega_n \to \mathbb{K}$ is the evaluation 
at $\vec{e}_i \in \Delta^n$\,. 

Due to \cite[Lemma 3.5]{Ezra-infty}, the operator $h^i_n$ satisfies the property
\begin{equation}
\label{h-in-square}
h^{i}_n \circ h^{i}_n=0\,.
\end{equation}

For a filtered $\sLie$-algebra $L$, a positive integer $n$, and 
an integer $0 \le i \le n$, we consider the following subspace
of  $\big(L \hotimes \Om_n \big)^0$
\begin{equation}
\label{stub}
\Stub^i_n(L)  : = \{ (\pa + d) \xi ~\big|~ \xi \in  \big(L \hotimes \Om_n \big)^{-1}, ~ \textrm{such that}~ \pa(\xi)\big|_{t_i =1} = 0 \}\,.
\end{equation}
For example, from \eqref{eq:ho_op2} we see that for every MC element $\al \in L  \hotimes \Om_n $, the element 
\begin{equation}
\label{diff-h-al}
(\pa + d) \circ h^i_n(\al) \in \Stub^i_n(L)\,.
\end{equation}

Given a pair $(\mu, \nu) \in \MC(L)  \times \Stub^i_n(L)$, we define 
the following sequence $\{\al^{(k)} \}_{k \ge 0}$ of degree $0$ elements in $L \hotimes \Om_n$
\begin{equation}
\label{Ezra-recursion}
\begin{array}{c}
\al^{(0)} : = \mu + \nu\,, \\[0.3cm] 
\displaystyle
\al^{(k+1)}  : = \al^{(0)} - \sum_{m=2}^{\infty} \frac{1}{m!} h^i_n  \, \{ \al^{(k)}, \dots,  \al^{(k)}\}_m \,.
\end{array}
\end{equation}
 
A simple inductive argument shows that $\alpha^{(k+1)} - \alpha^{(k)} \in
\cF_{k+2} L \hotimes \Omega_n$. Hence, the sequence converges to
$\alpha = \lim \alpha^{(k)}$, which satisfies
\begin{equation} 
\label{eq:ezra_limit}
\alpha =  \mu + \nu - \sum_{m \geq 2} \frac{1}{m!} h^{i}_{n}
\{ \al, \dots,  \al\}_m\,.
\end{equation}
Combining \eqref{eq:ho_op2} with \eqref{eq:ezra_limit}
and the definition of $\Stub^i_n(L)$, we see that 
\begin{equation}
\label{back-to-mu}
\ve^{n}_{i} ( \alpha ) = \mu\,. 
\end{equation}
Furthermore, combining \eqref{eq:poincare} with \eqref{h-in-square} and \eqref{eq:ezra_limit}, we deduce that  
\begin{equation}
\label{back-to-nu}
(\pa + d) \circ h^i_n (\alpha) = \nu\,.
\end{equation}

Let us now show that $\alpha$ is a MC element of $L \hotimes \Om_n$.
Using \eqref{eq:poincare}, \eqref{back-to-mu} and equation \eqref{Bianchi}
from Proposition \ref{prop:curv} we see that
\begin{equation} 
\label{eq:curvalp1}
\begin{split}
\curv(\alpha) &= \curv(\mu) + \sum_{m \geq 2} \frac{1}{m!} h^{i}_{n} (\pa + d) 
\{ \al, \dots,  \al\}_m \,\\
& = h^{i}_{n} (\pa + d) \curv(\alpha)\\
& = -\sum_{m \geq 1} \frac{1}{m!} h^{i}_{n} \{ \al, \dots,
\al, \curv(\alpha)\}_{m+1} \,.
\end{split}
\end{equation}
Therefore
$$
\curv(\alpha)  \in \cF_q L \hotimes \Om_n
$$
for all $q \ge 1$ and hence $\al$ is indeed a MC element of $L \hotimes \Om_n$\,.

Let us now prove that every MC element $\al$ of  $L \hotimes \Om_n$ is 
determined uniquely by the pair  $(\mu, \nu) \in \MC(L)  \times \Stub^i_n(L)$\,, 
where 
$$
\mu = \ve^i_n(\al)\,, \qquad \textrm{and} \qquad \nu = (\pa + d) \circ h^i_n (\alpha)\,.
$$

Indeed, if $\beta$ is another MC element of $L \hotimes \Om_n$ such that
\[
\ve^{i}_n \alpha = \ve^{i}_n \beta, \qquad \textrm{and} \qquad
(\pa + d) \circ h^i_n \alpha = (\pa + d) \circ h^i_n \beta
\]
then, using \eqref{eq:poincare}, we see that
\[
\alpha -\beta = - \sum_{m \geq 2} \frac{1}{m!} h^{i}_n \bigl( \{
\al, \dots,  \al\}_m -\{ \beta, \dots,  \beta\}_m \bigr)
\]
which means that 
$$
\alpha -\beta \in \cF_q L \hotimes \Om_n
$$
for all $q \ge 1$ and hence $\al = \beta$\,.

Combining this observation with the above consideration of the limiting 
element of sequence \eqref{Ezra-recursion}, we get the following 
version\footnote{In the definition of the corresponding subspace $\Stub^i_n(L)$ \eqref{stub}
in paper \cite{Ezra-infty}, the condition $\pa(\xi)\big|_{t_i =1} = 0$ is probably omitted by mistake. 
We are sure that this condition is unavoidable.} 
of \cite[Lemma 4.6]{Ezra-infty}:
\begin{lem}
\label{lem:Ezra}
Let $n \ge 1$ and $0\le i \le n$. 
For every filtered $\sLie$-algebra $L$ the assignment 
\begin{equation}
\label{to-stub}
\al \mapsto \big( \varepsilon^{n}_{i}\al  ~, ~(\pa + d) \circ h^i_n(\al)  \big)
\end{equation}
is a bijection of sets 
$$
\MC(L  \hotimes \Om_n )  \cong \MC(L)  \times \Stub^i_n(L)\,.
$$
The inverse of \eqref{to-stub} assigns to a pair 
$(\mu, \nu) \in \MC(L)  \times \Stub^i_n(L)$ the limiting element 
of sequence  \eqref{Ezra-recursion}. \qed
\end{lem}

\section{Every $1$-cell can be ``rectified''} 
\label{app:rectify}
Let $L$ be a filtered $\sLie$-algebra and $\mMC_{\bul}(L)$ be the 
corresponding DGH $\infty$-groupoid. 
In some cases, it is convenient to deal with $1$-cells 
in  $\mMC_{\bul}(L)$ of the form  
\begin{equation}
\label{rectified}
\beta = \beta_0(t_0) + d t_0  \, \beta_1\,,  
\end{equation}
where $\beta_1$ is a vector in $L^{-1} \subset L^{-1}  \hotimes \bbk[t_0]$, i.e. the 
component $\beta_1$ does not ``depend'' on $t_0$.

We call such $1$-cells {\it rectified}\footnote{Note that rectified $1$-cells are 
precisely $1$-cells in the Kan complex $\ga_{\bul}(L)$ introduced in \cite[Section 5]{Ezra-infty}.}.  

\begin{remark}
\label{rem:rectified}
It is not hard to see that under the bijection established 
in Lemma \ref{lem:Ezra} rectified $1$-cells \eqref{rectified} in  $\mMC_{\bul}(L)$
correspond to pairs $(\mu, \nu)\in \MC(L) \times \Stub^1_1(L)$ where 
$$
\mu =  \beta_0(0) \qquad \textrm{and} \qquad
\nu = (\pa + d) (t_0 \beta_1)\,.
$$
\end{remark}

In this appendix we prove the following lemma: 
\begin{lem}
\label{lem:rectify}
Let $L$ be a filtered $\sLie$-algebra $\al = \al_0(t) + d t  \, \al_1(t) $
be a $1$-cell of $\mMC_{\bul}(L)$ such that $\al_1(t) \in \cF_k L^{-1} \hotimes \bbk[t]$
for some $k \ge 1$. Then there exists $1$-cell of $\mMC_{\bul}(L)$
\begin{equation}
\label{beta-cell-constant}
\beta = \beta_0(t) + d t  \, \beta_1 
\end{equation}
with $\beta_1 \in \cF_k L^{-1}$ such that 
\begin{equation}
\label{ends-OK}
\beta_0(0) = \al_0(0) \qquad \textrm{and} \qquad 
\beta_0(1) = \al_0(1)\,.
\end{equation}
\end{lem}
\begin{proof}
We will prove by induction that there exists a sequence of 
MC elements  $(m \ge k)$
\begin{equation}
\label{ga-m}
\ga^m  \in \MC(L \hotimes \Om_2 ) 
\end{equation}
such that 
\begin{equation}
\label{0-face}
\ga^m \big |_{t_0 = 0} =   \al_0(t_2) + d t_2  \, \al_1(t_2)\,,
\end{equation}
the $1$-cell 
\begin{equation}
\label{2-face}
\beta^m = \beta^m_0(t_0) + d t_0 \beta^m_1 : =  \ga^m \big |_{t_2 = 0} 
\end{equation}
is rectified and\footnote{For $m=k$ we will actually have 
$\beta^k_1 \equiv 0 $\,.}
\begin{equation}
\label{beta-1-m-OK}
\beta^m_1 \in \cF_{k} L\,, \qquad \forall ~~ m \ge k\,.
\end{equation}
Finally the $1$-cell 
\begin{equation}
\label{sigma-m}
\si^m_0(t_2) + d t_2\, \si^m_1(t_2) : =  \ga^m \big|_{t_1 = 0}  
\end{equation}
satisfies the condition
\begin{equation}
\label{sigma-deeper}
\si^m_1(t_2) \in  \cF_{m} L \hotimes \bbk[t_2]
\end{equation}
and, for every $m \ge k$, we have 
\begin{equation}
\label{deeper}
\ga^{m+1} - \ga^m \in  \cF_{m} L \hotimes  \Om_2\,.
\end{equation}

The links between the $2$-cell $\ga^m$ and $1$-cells 
$ \al_0(t_2) + d t_2  \, \al_1(t_2)$, $\beta^m_0(t_0) + d t_0 \, \beta^m_1 $, 
$\si^m_0(t_2) + d t_2\, \si^m_1(t_2)$ are presented graphically on figure 
\ref{fig:ga-cell}. 
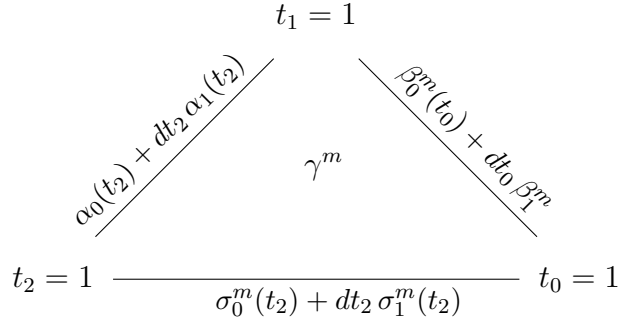
\begin{figure}[htp]
\centering 
\begin{tikzpicture}[scale=0.5]
\tikzstyle{vert}=[circle, minimum size=3, inner sep=5]
%%%%
\node[vert] (v1) at (0, 7) {{$ t_1 = 1$}};
\node[vert] (v0) at (7, 0) {{$ t_0 = 1$}};
\node[vert] (v2) at (-7, 0) {{$ t_2 = 1$}};
\draw (0.2,3) node[anchor=center] {{$\ga^m$}};
% edges
\draw (v1) edge (v0);
\draw (4.2,3.7) node[anchor=center, rotate = -45] {{\small $\beta^m_0(t_0) + d t_0 \, \beta^m_1 $}};
\draw (v1) edge (v2);
\draw (-4.2,3.7) node[anchor=center, rotate = 45] {{\small $ \al_0(t_2) + d t_2  \, \al_1(t_2)$}};
\draw (v2) edge (v0);
\draw (0.6,-0.6) node[anchor=center] {{\small $\si^m_0(t_2) + d t_2\, \si^m_1(t_2)$}};
\end{tikzpicture}
\caption{Links between different cells} 
\label{fig:ga-cell}
\end{figure}

The $1$-cell $\al$ gives us the following MC element 
\begin{equation}
\label{ga-k}
\ga^k : = \al_0(t_2) + d t_2  \, \al_1(t_2) 
\end{equation}
in $L \hotimes \Om_2$ which satisfies all the 
above conditions. So this is the base of our induction. 

Let us assume that we constructed MC elements $\ga^q$ satisfying 
the above properties for all $k \le q \le m$. Now we will use $\ga^m$ to 
construct $\ga^{m+1}$.

Let us denote by $\rho^{(0)}$ the following degree $0$ element 
in $L \hotimes \Om_2$: 
\begin{equation}
\label{rho-0}
\rho^{(0)} : = \al_0(0) + (\pa + d) \circ h^1_2 (\ga^m)\,.
\end{equation}

According to Lemma \ref{lem:Ezra}, $\ga^m$ is the limiting element of the 
convergent sequence $\{ \rho^{(p)}\}_{p \ge 0}$ defined inductively by 
\begin{equation}
\label{rho-p-p1}
\rho^{(p+1)} : = \rho^{(0)} - \sum_{l=2}^{\infty} \frac{1}{l!} h^1_2 \, \{\rho^{(p)}, \dots ,\rho^{(p)}\}_l \,.
\end{equation}

Let $\{\ti{\rho}^{(p)}\}_{p \ge 0}$ be the following sequence 
of degree $0$ elements in $L \hotimes \Om_2$:
\begin{equation}
\label{ti-rho-0}
\ti{\rho}^{(0)} : = \al_0(0) + (\pa + d) \big( \D + h^1_2 (\ga^m) \big)\,,
\end{equation}
\begin{equation}
\label{ti-rho-p-p1}
\ti{\rho}^{(p+1)} : = \ti{\rho}^{(0)} - \sum_{l=2}^{\infty} \frac{1}{l!} h^1_2 \, \{\ti{\rho}^{(p)}, \dots , \ti{\rho}^{(p)}\}_l \,,
\end{equation}
where $\D$ is a degree $-1$ element in $ \cF_m L^{-1} \otimes \bbk[t_0, t_2] $
of the form 
\begin{equation}
\label{Delta}
\D  : =  \sum_{s = 0}^{N} \D_s t_0 t_2^s \,.
\end{equation}

Let us denote by $\ga^{m+1}$ the limiting element of the 
sequence  $\{\ti{\rho}^{(p)}\}_{p \ge 0}$ and show that, for 
an appropriate choice of coefficients $\D_s$ in \eqref{Delta}
the element $\ga^{m+1}$ satisfies all the desired properties.  

Since $(\pa + d) \D$ disappears on the $0$-th face,
$$
\ga^{m+1} \big|_{t_0= 0} = \ga^{m} \big|_{t_0= 0} =  \al_0(t_2) + d t_2  \, \al_1(t_2)\,.
$$ 

We also have 
$$
(\pa + d)\D \big|_{t_2 = 0} ~ = ~ (\pa + d) (\D_0 t_0)\,.
$$
Hence, due to Remark \ref{rem:rectified}, the restriction of 
$\ga^{m+1}$ to the second face is a rectified $1$-cell 
$$
\beta^{m+1}_0(t_0) + d t_0 \beta^{m+1}_1\,.   
$$
Inclusion $\beta^{m+1}_1  \in \cF_k L$ holds because 
$\D_0 \in  \cF_m L$ with $m \ge k$\,. 

Since $\D  \in \cF_m L \otimes  \bbk[t_0, t_2]$, condition \eqref{deeper} is 
obviously satisfied. Moreover, 
\begin{equation}
\label{ga-m1-ga-m}
\ga^{m+1} - \ga^m - (\pa + d) \, \D \in   \cF_{m+1} L \hotimes  \Om_2 \,.
\end{equation}

Hence for 
$$
\si^{m+1} = \si^{m+1}_0(t_2) + d t_2\, \si^{m+1}_1(t_2) : =  \ga^{m+1} \big|_{t_1 = 0}  
$$
we have\footnote{Observe that when $t_1=0$, $t_0=1-t_2$.} 
\begin{equation}
\label{si-m1-si-m}
\si^{m+1}_1(t_2) - \si^{m}_1(t_2) - \frac{\pa}{\pa t_2} \sum_{s = 0}^{N} \D_s (1-t_2) t_2^s  ~\in~    \cF_{m+1} L \hotimes \bbk[t_2] \,.
\end{equation}

Since  $\si^m_1(t_2) \in   \cF_{m} L \hotimes \bbk[t_2]$, we have 
\begin{equation}
\label{si-1-expansion}
\si^m_1(t_2) = \sum_{s = 0}^{N} \si^m_{1, s} \, t_2^s + \dots\,,
\end{equation}
where each $ \si^m_{1, s} \in  \cF_{m} L$ and $\dots$ denotes terms in 
$\cF_{m+1} L \hotimes \bbk[t_2]$\,.

Therefore,  using inclusion \eqref{si-m1-si-m} and obvious identity 
\begin{equation}
\label{si-si-si-Del}
\frac{\pa}{\pa t_2} \sum_{s = 0}^{N} \D_s (1-t_2) t_2^s = 
 \sum_{s=0}^N (s+1) (\D_{s+1} - \D_s) t^s_2
\end{equation}
with $\D_{N+1} = 0$, we conclude that 
$$
\si^{m+1}_1(t_2) \in \cF_{m+1} L \hotimes \bbk[t_2]
$$ 
provided the coefficients $\{ \D_s \}_{0 \le s \le N}$ satisfy the linear equations: 
\begin{equation}
\label{lin-system-Del}
\D_s- \D_{s+1} = \frac{ \si^m_{1, s}}{s+1}\,, \qquad 0 \le s \le N\,, 
\qquad \D_{N+1}=0\,.
\end{equation}

Since linear system \eqref{lin-system-Del} has the obvious solution
\begin{equation}
\label{solution-Del}
\D_{N} = \frac{ \si^m_{1, N}}{N+1}\,,  \qquad \textrm{and} \qquad
\D_s = \frac{ \si^m_{1, s}}{s+1} + \D_{s+1}\,, \qquad \forall ~~ 0 \le s \le N-1
\end{equation}
we conclude that the desired MC element $\ga^{m+1}$ can be constructed. 

Let us now show that the existence of sequence \eqref{ga-m} satisfying all the 
above conditions implies the lemma.

Indeed, due to inclusions \eqref{sigma-deeper} and \eqref{deeper} the sequence $\{\ga_m\}_{m \ge k}$ 
converges to a MC element 
$$
\ga \in L \hotimes \Om_2 
$$
for which the $1$-cell
$$
\si_0(t_2) + d t_2 \si_1(t_2) : = \ga \big|_{t_1 = 0} 
$$
has the property $\si_1(t_2) \equiv 0$\,.

Therefore, we have 
$$
\al_0(1) = \ga \big|_{t_2 = 1} =  \ga \big|_{t_0 = 1} = \beta_0(1)\,.
$$

Thus, since $\beta_0(0) = \al_0(0)$ by construction, Lemma \ref{lem:rectify} 
follows. 
\end{proof}

\begin{remark}
\label{rem:from-Ezra}
Lemma \ref{lem:rectify} can be deduced from \cite[Corollary 5.11]{Ezra-infty}
using twisting by the MC element $\al_0$ and some basic facts proved in 
\cite[Section VI]{Goerss-Jardine}. However, we decided to give a 
proof which bypasses the introduction of an additional  
Kan complex $\ga_{\bul}(L)$ from  \cite[Section 5]{Ezra-infty}. 
\end{remark}

\section{Concatenating infinitely many 1-cells}
\label{app:con-ing}

Let $L$ be a filtered $\sLie$-algebra and  
\begin{equation}
\label{app-rho-n-seq}
\big\{   \rho^{(n)}  = \rho^{(n)}_0(t_0) + d t_0  \rho^{(n)}_1(t_0)   \big\}_{n \geq 1}
\end{equation}
be a sequence of $1$-cells in $\mMC_{\bul}(L)$ satisfying the following conditions: 
\begin{equation}
\label{app-rho-n-1}
\rho^{(n)}_1(t_0) \in  \cF_{n} L \hotimes \bbk[t_0] \,, 
\end{equation}
and
\begin{equation}
\label{tails-match}
\rho^{(n)}_0(1) = \rho^{(n+1)}_0(0)\,, \qquad \forall~~n \ge 1\,. 
\end{equation}

The following lemma implies that there exists a $1$-cell
in  $\mMC_{\bul}(L)$ which may be viewed as the result of 
concatenating the sequence of $1$-cells \eqref{app-rho-n-seq}. 

\begin{lem}
\label{lem:con-ing}
Under the above conditions on $1$-cells \eqref{app-rho-n-seq}, 
one can construct a sequence of $1$-cells 
\begin{equation}
\label{ga-n-seq}
\big\{ \ga^{(n)} =  \ga^{(n)}_0(t_0) + d t_0  \ga^{(n)}_1 (t_0) \big\}_{n \ge 1}
\end{equation}
in $\mMC_{\bul}(L)$ such that 
\begin{equation}
\label{ga-n-ends}
\ga^{(n)}_0(0) =  \rho^{(1)}_0(0)\,, \qquad 
\ga^{(n)}_0(1) =  \rho^{(n)}_0(1)   \,, 
\end{equation}
and 
\begin{equation}
\label{ga-n-ga-n1}
\ga^{(n+1)} -  \ga^{(n)} \in \cF_{n+1} L \hotimes \Om_1\,. 
\end{equation}
\end{lem}
\begin{proof}
Setting 
$$
\ga^{(1)}  : = \rho^{(1)}
$$
we get the base of the induction on $n$. So let us assume that 
we already constructed the desired $1$-cells 
$$
\ga^{(1)}, \ga^{(2)}, \dots, \ga^{(n)}\,.
$$

Since 
$$
\ga^{(n)}_0(1) = \rho^{n+1}_0(0)\,,
$$
the $1$-cells $\ga^{(n)}$ and $\rho^{(n+1)}$ give us the following horn: 
\begin{equation}
\label{1-dim-horn-L}
\begin{tikzpicture}
\matrix (m) [matrix of math nodes, row sep=3em, column sep=2em]
{~  &  \mathbf{1}  \mapsto \rho^{(n+1)}_0(0) &  ~ \\
\mathbf{2} \mapsto \rho^{(1)}_0(0)  &  ~  &   \mathbf{0} \mapsto  \rho^{(n+1)}_0(1)\\ };
\path[dashed, font=\scriptsize]
(m-1-2) edge  node[above] {$\ga^{(n)}~~~$} (m-2-1)  edge node[auto] {$\rho^{(n+1)}$}  (m-2-3);
\end{tikzpicture}
\end{equation}

Let us prove that there exists a MC element
\begin{equation}
\label{eta}
\eta  \in L \hotimes \Om_2 
\end{equation}
such that 
\begin{equation}
\label{eta-rho}
\eta \Big|_{t_2 = 0}  =  \rho^{(n+1)}_0(t_0) + d t_0  \rho^{(n+1)}_1(t_0) \,,
\end{equation}
\begin{equation}
\label{eta-gamma}
\eta \Big|_{t_0 = 0}  =  \ga^{(n)}_0(t_1) + d t_1  \ga^{(n)}_1(t_1) \,,
\end{equation}
and the $1$-cell 
\begin{equation}
\label{ga-n1}
\ga^{(n+1)}_0(t_0) + d t_0  \ga^{(n+1)}_1(t_0) : = \eta \Big|_{t_1 = 0}
\end{equation}
satisfies the property
\begin{equation}
\label{ga-n1-ga-n}
\ga^{(n+1)}  -   \ga^{(n)}  \in  \cF_{n+1} L \hotimes \Om_1\,.
\end{equation}

To construct $\eta$, we denote by $\ga$ and $\rho$ the degree $-1$ elements 
\begin{equation}
\label{ga}
\ga (t_0) : =  h^0_1 (  \ga^{(n)} ) \in  L \hotimes \bbk[t_0]
\end{equation} 
and 
\begin{equation}
\label{rho}
\rho (t_0) : =  h^1_1 (  \rho^{(n+1)} ) \in  L \hotimes \bbk[t_0]
\end{equation} 
respectively. 

In other words, $(\pa + d) \ga (t_0)$ (resp.  $(\pa + d) \rho (t_0)$)
is the ``stub'' of the $1$-cell  $\ga^{(n)}$ (resp. $\rho^{(n+1)}$) 
in the sense of Lemma \ref{lem:Ezra}.

Next, we consider the following degree $0$ element of $L \hotimes \Om_2$
\begin{equation}
\label{nu}
\nu : = (\pa + d) \big(\ga(1-t_2) +  \rho(t_0)   \big)\,.
\end{equation}
The element $\nu$ belongs to $\Stub_2^1(L)$ since, at vertex $1$, $1- t_2 = 1$ and $t_0=0$\,.

Therefore, according to Lemma \ref{lem:Ezra}, the sequence of 
degree zero elements 
$$
\{\eta^{(k)}\}_{k \ge 0}
$$
defined by
\begin{equation}
\label{eta-0}
\eta^{(0)} : =  \rho^{(n+1)}_0(0)   + \nu
\end{equation}
\begin{equation}
\label{eta-k1-eta-k}
\eta^{(k+1)} : = \eta^{(0)} - \sum_{m=2}^{\infty} \frac{1}{m!} \{\eta^{(k)}, \dots  ,\eta^{(k)} \}_m
\end{equation}
converges to a MC element $\eta \in L \hotimes \Om_2$\,.
 
Let us prove that the limiting element $\eta$ satisfies all the desired properties.

Indeed, \eqref{eta-rho} and \eqref{eta-gamma} follow from Lemma \ref{lem:Ezra} 
and the obvious equations 
$$
\nu \Big|_{t_2 = 0} =  (\pa + d) \big( \rho(t_0) \big)\,, 
\qquad \textrm{and} \qquad  
\nu \Big|_{t_0 = 0} = (\pa + d) \big( \ga(1-t_2) \big) =  (\pa + d) \big( \ga(t_1) \big)\,.
$$

Inclusion \eqref{app-rho-n-1} implies that 
\begin{equation}
\label{nu-approx}
\nu ~ - ~(\pa + d) \big( \ga(1-t_2) \big) \in \cF_{n+1} L \hotimes \Om_2\,.
\end{equation}

On the other hand a direction computation shows that
$$
h^1_2 \big( \ga^{(n)}_0(1-t_2) - d t_2  \ga^{(n)}_1(1 - t_2) \big) = \ga(1-t_2)\,,
$$
where $\ga^{(n)}_0(1-t_2) - d t_2  \ga^{(n)}_1(1 - t_2)$ is viewed as an element of 
$L \hotimes \Om_2$ in the obvious way.

Hence, applying Lemma \ref{lem:Ezra} to MC elements of 
$(L / \cF_{n+1} L) \otimes \Om_2 $, and using \eqref{nu-approx} 
we conclude that 
$$
\eta - \ga^{(n)}_0(1-t_2) + d t_2  \ga^{(n)}_1(1-t_2) \in \cF_{n+1} L \hotimes \Om_2\,.
$$

Thus $1$-cell \eqref{ga-n1} indeed satisfies \eqref{ga-n1-ga-n}. 

Lemma \ref{lem:con-ing} is proved. 
\end{proof}

\pagebreak
~\\

\noindent\textsc{Department of Mathematics,
Temple University, \\
Wachman Hall Rm. 638\\
1805 N. Broad St.,\\
Philadelphia PA, 19122 USA \\
\emph{E-mail address:} {\bf vald@temple.edu} }

~\\

\noindent\textsc{Mathematics Institute\\
Georg-August Universit\"at G\"ottingen, \\
Bunsenstrasse 3-5, D-37073 G\"ottingen, Germany\\
\emph{E-mail address:} {\bf crogers@uni-math.gwdg.de} }

\end{document}